\documentclass[preprint,15pt]{elsarticle}

\usepackage{amssymb}
\usepackage{amsmath}
\usepackage{amsthm}
\usepackage{amssymb,latexsym}
\usepackage{mathrsfs}
\usepackage{enumerate}
\usepackage{indentfirst}
\usepackage{color}
\usepackage{graphicx}
\usepackage{subfigure}

\newtheorem{definition}{Definition}
\newtheorem{example}{Example}
\newtheorem{theorem}{Theorem}
\newtheorem{lemma}{Lemma}
\newdefinition{remark}{Remark}
\newtheorem{corollary}{Corollary}

\journal{SIAM Journal on Matrix Analysis and Applications}

\begin{document}
\begin{frontmatter}


\title{The Role of Tensor-Generated Matrices in Analyzing Spin State Classicality and Tensor H-Eigenvalue Distributions }

\author{Liang Xiong$^{1,}$\corref{mycorrespondingauthor}}
\ead{liang.xiong@polyu.edu.hk}

\author{ Jianzhou Liu$^{2,}$\corref{mycorrespondingauthor}}
\cortext[mycorrespondingauthor]{Corresponding author}
\ead{liujz@xtu.edu.cn}

\address{$^{1}$ Department of Applied Mathematics, The Hong Kong Polytechnic University, Hung Hom, Hong Kong, China\\
$^{2}$School of Mathematics and Computational Science, Xiangtan University, Xiangtan, Hunan 411105, China\\}

\begin{abstract}
Multipartite quantum scenarios are a significant and challenging resource in quantum information science. Tensors provide a powerful framework for representing multipartite quantum systems. In this work, we introduce the role of tensor-generated matrices that can broadly be defined as the relationships between an $m$-th order $n$-dimensional tensor and an $n$-dimensional square matrix. Through these established connections, we demonstrate that the classification of the tensor-generated matrix as an $H$-matrix implies the original tensor is also an $H$-tensor.  We also explore various similar properties exhibited by both the original tensor and the tensor-generated matrix, including weak irreducibility, weakly chained diagonal dominance, and (strong) symmetry. These findings provide a method to transform intricate tensor problems into matrices in specific contexts, which is especially pertinent due to the NP-hard complexity of the majority of tensor problems. Subsequently, we explore the application of tensor-generated matrices in analyzing the classicality of spin states. Leveraging the tensor representation, we introduce classicality criteria for (strongly) symmetric spin-$j$ states, which potentially provide fresh perspectives on the study of multipartite quantum resources. Finally, we extend classical matrix eigenvalue inclusion sets to higher-order tensor $H$-eigenvalues, a task that is typically challenging for higher-order tensors. Consequently, we propose representative tensor $H$-eigenvalue inclusion sets, such as modified Brauer's Ovals of Cassini sets, Ostrowski sets, and $S$-type inclusion sets.
\end{abstract}


\end{frontmatter}

\section{Introduction}

Tensors, also known as multidimensional arrays, and their eigenvalues have become increasingly important in the fields of applied mathematics and computational mathematics, fostering the development of numerical multilinear algebra. Moreover, tensors find wide-ranging applications in various domains, including liquid crystals \cite{CYN2017L1,liquid2,liquid3,East}, higher-order Markov chains \cite{Markov2}, best-rank one approximation in data analysis \cite{data1,data2,KE2016}, quantum information \cite{GME2003,hu2016,quantuminformation,COAM,ACAP}, and the positive definiteness of even-order multivariate forms in automatic control \cite{Niq2008,JMI}.

With the progress in technology and the evolution of quantum platforms, many platforms are now capable of concurrently generating and manipulating multipartite qubits. It has become evident in recent years that high-dimensional systems offer advantages in specialized tasks such as quantum communication \cite{PhysRevLett.88.127902,PhysRevA.82.030301} and quantum computing \cite{PhysRevLett.83.5162,RevModPhys.80.1061}. Consequently, an increasing amount of multi-order high-dimensional tensor data will manifest in multipartite quantum systems.

Tensors provide a robust framework for representing complex quantum data structures \cite{chen2010tensor,chitambar2008tripartite,bruzda2024rank,xiong2022cauchy}. Analyzing the eigenvalues of these tensors uncovers crucial insights into quantum system properties and behaviors \cite{GME2003,hu2016}. This exploration is essential for understanding multipartite quantum states \cite{xiong2022geometric}, entanglement phenomena, and optimizing quantum computational tasks \cite{PRXQuantum.2.030305}. By deciphering the properties of tensors \cite{2014T,PhysRevA.94.042324}, we can reveal intricate quantum correlations, leading to advancements in quantum algorithms and communication protocols. This research not only enriches the theoretical foundations of quantum information science but also drives innovations in quantum technologies, enabling more efficient information processing and communication capabilities in various practical applications.

The coherent superposition of states \cite{RevModPhys.89.041003}, along with observable quantization, stands out as a key characteristic signifying the distinction of quantum mechanics from classical physics. Quantum coherence within many-body systems encapsulates the core nature of entanglement and serves as a crucial element for numerous physical phenomena in quantum optics \cite{optics,optics2}, quantum information \cite{RevModPhys.89.041003}, solid-state physics \cite{gao2015coherent}, and nanoscale thermodynamics \cite{PhysRevLett.111.250404}. Recently, there has been a growing interest in studying the existence and functional significance of quantum coherence \cite{PhysRevLett.113.140401,PhysRevLett.116.150502,PhysRevA.98.022328}. Despite its fundamental significance, the establishment of a comprehensive theory regarding quantum coherence as a physical asset has only commenced in recent times.

Spin coherent states are the ones that come as close as possible to the ideal of a classical phase-space point, in the sense that their quantum fluctuations for the angular
moment components are as small as allowed by Heisenbergs uncertainty relation \cite{coherent}, which is widely used in quantum optics \cite{optics,optics2} and other fields \cite{1986Generalized}. The spin-$j$ states is classical \cite{Classicality2008,PhysRevA.96.032312} if its density matrix can be decomposed as a weighted sum of angular momentum coherent states with positive weights, that is,
the set of classical spin states are defined as the convex hull of spin coherent states.
Quantum states that lie beyond the domain of classical spin states are regarded as nonclassical (truly quantum). The quantumness \cite{quantumness2010,quantumness2016} of a spin-$j$ state, used to quantify the separation between a quantum state and a classical one, can be gauged by its distance from the classical states. Notably, employing various distance metrics will yield different measures of quantumness, including the Hilbert-Schmidt distance \cite{quantumness2010}, Bures distance \cite{2010H}, and trace distance \cite{Eisert_2003}.


Despite tensors' superior descriptive capacity for multi-order high-dimensional data compared to matrices, a notable challenge arises from the prevalence of NP-hard problems in tensor analysis \cite{hillar2013most}. In contrast to matrix eigenvalue computations, determining eigenvalues and confirming positive semidefiniteness in higher-order tensors is acknowledged to be NP-hard. Despite ongoing efforts, various algorithms have emerged for calculating tensor eigenvalues, as detailed in \cite{Power2011,Power2014,Cui2014,Homotopy1,Homotopy2}. Regrettably, these methods exhibit limitations in handling larger tensors.

Although the Ger$\check{s}$gorin disk theorem can be directly extended to higher-order tensors, classical matrix eigenvalue inclusion sets such as Brauer's Ovals of Cassini sets, Ostrowski sets, and $S$-type inclusion sets cannot be directly extended to higher-order tensors. However, these classical matrix eigenvalue inclusion sets often capture all eigenvalues more accurately than Ger$\check{s}$gorin disk sets. So far, most tensor eigenvalue inclusion sets have focused on modified Brauer-type sets, as Brauer's Ovals of Cassini sets \cite{BrauerA} cannot be directly extended to tensor eigenvalues. Furthermore, matrices Ostrowski sets \cite{OstrowskiI}, which depend not only on deleted row sums $r_{i}(A)$ but also on deleted column sums $c_{i}(A)$, are generally considered unsuitable for higher-order tensors with multiple indices. Clearly, compared to the matrix eigenvalue localization theory, the theory of higher-order tensor eigenvalue localization is still in its early stages and requires further development.


In this paper, we introduce the role of tensor-generated matrices and their applications in analyzing spin state classicality and tensor $H$-Eigenvalue distributions. Firstly, the tensor-generated matrices can broadly be defined as the relationships between an $m$-th order $n$-dimensional tensor $\mathcal{A}$ and an $n$-dimensional square matrix $A$. Specifically, this classification allows for aligning the elements of tensor $\mathcal{A}$ with the corresponding elements of matrix $A$ of identical dimensions.
Through these established relationships, we show that when the tensor-generated matrix is classified as an $H$-matrix, it necessarily infers that the original tensor is an $H$-tensor.
We also investigate several analogous properties shared by the original tensor and the tensor-generated matrix, such as weak irreducibility, weakly chained diagonal dominance and (strong) symmetry.
The findings will facilitate an extensive exploration of $H$-tensor theory by leveraging the well-established $H$-matrix theory. Additionally, they offer a methodology to convert complex tensor issues into matrices in specific scenarios, particularly relevant given the NP-hard nature of most tensor problems.



Subsequently, we investigate the use of tensor-generated matrices for analyzing the classicality of spin states. The preceding analysis establishes ample criteria for determining the semi-positive definiteness of even-order symmetric tensors. Drawing on the tensor representation of spin states and the tensor-generated matrix, we introduce specific classicality criteria for spin-$j$ states. Notably, an important observation arises: when the original tensor displays strong symmetry, the resulting tensor-generated matrix also exhibits symmetry. Consequently, we derive classicality criteria for spin-$j$ states characterized by strongly symmetric coefficient tensors.

Finally, by leveraging the established connections between tensors and tensor-generated matrices, we extend classical matrix eigenvalue inclusion sets to higher-order tensor $H$-eigenvalues, a task that is typically challenging for higher-order tensors. This extension is enabled by the relationship between the $H$-matrix subclass and matrix eigenvalue distribution. Consequently, we propose representative tensor eigenvalue inclusion sets, such as modified Brauer's Ovals of Cassini sets, modified Ostrowski sets, and modified $S$-type inclusion sets of tensor $H$-eigenvalues.

The paper is organized as follows: Section 2 provides the necessary preliminaries for tensor and quantum information theory. In Section 3, we introduce the concept of the tensor-generated matrix and establish the relationship between an $H$-matrix tensor-generated matrix and an $H$-tensor. We also explore similar properties shared by the original tensor and the tensor-generated matrix, such as weak irreducibility and weakly chained diagonal dominance.  In section 4, we explore the tensor-generated matrix application in spin states, and present some classical criteria for (strongly) symmetric spin-$j$ states. In Section 5, based on the established connections, we present representative tensor eigenvalue inclusion sets utilizing the connection between the $H$-matrix subclass and matrix eigenvalue distribution, including modified Brauer's Ovals of Cassini sets, modified Ostrowski sets, and modified $S$-type inclusion sets.

\section{Preliminaries}
\label{sec:pre}


In this section, we begin by introducing key definitions and established results regarding $H$-tensors and quantum information. It is important to note the consistent notation employed throughout this paper: lowercase letters, such as $x$, $y$, etc., represent vectors; italic capital letters, such as $A$, $B$, etc., represent matrices; and calligraphic capital letters, such as $\mathcal{A}$, $\mathcal{B}$, etc., denote tensors.

\subsection{Preliminaries for tensor}

Let $\mathbb{C}(\mathbb{R})$ represent the set of all complex (real) numbers, and let $N={1,2,\cdots,n}$ denote a positive integer $n\geq 2$. We define a complex (real) tensor $\mathcal{A}=(a_{i_1i_2\cdots i_m})$ of order $m$ and dimension $n$ as $\mathbb{C}^{[m\times n]}(\mathbb{R}^{[m\times n]})$, where the entries satisfy
$a_{i_1i_2\cdots i_m}\in \mathbb{C}(\mathbb{R}),$
with $i_j=1,2,\cdots,n$ for $j=1,2,\cdots,m$. Notably, a vector corresponds to a tensor of order 1, and a matrix corresponds to a tensor of order 2. For a tensor $\mathcal{A}=(a_{i_1i_2\cdots i_m})\in \mathbb{R}^{[m\times n]}$, $\mathcal{A}$ is considered nonnegative if every entry satisfies $a_{i_1i_2\cdots i_m}\geq 0$. A real tensor $\mathcal{A}=(a_{i_1i_2\cdots i_m})$ is called symmetric if
$$a_{i_1i_2\cdots i_m}=a_{\pi(i_1i_2\cdots i_m)}, \forall \pi\in\Pi_{m},$$
where $\Pi_{m}$ represents the permutation group of $m$ indices. A unit tensor, denoted by $\mathcal{I}$, has all diagonal entries 1 and otherwise entries 0.

In 2005, Qi \cite{Qi2005} and Lim \cite{Lim} independently introduced the concept of eigenvalues for higher-order tensors.
Let $\mathcal{A}=(a_{i_1i_2\cdots i_m})$ be a complex tensor of order $m$ and dimension $n$, and let $x=(x_{1},x_{2},\cdots,x_{n})^{\top}$ be an $n$-dimensional vector. The expression $\mathcal{A}x^{m-1}$ represents an $n$-dimensional vector in $\mathbb{C}^{n}$, where its $i$-th component is defined as
\begin{eqnarray*}
(\mathcal{A}x^{m-1})_i=\sum\limits_{i_2,i_3,\cdots,i_m=1}^{n}a_{ii_2\cdots i_m}x_{i_2}\cdots x_{i_m}.
\end{eqnarray*}

If there exist a complex number $\lambda$ and a nonzero complex vector $x=(x_{1},x_{2},\cdots,x_{m})^{\top}$ satisfying
\begin{equation}\label{eig}
\mathcal{A}x^{m-1}=\lambda x^{[m-1]},
\end{equation}
then $\lambda$ is referred to as an $N$-eigenvalue of $\mathcal{A}$, and $x$ is an eigenvalue of $\mathcal{A}$ associated with $\lambda$, where
$x^{[m-1]}=(x_{1}^{m-1},x_{2}^{m-1},\cdots,x_{n}^{m-1})^{\top}.$
If both $\lambda$ and $x$ are real, then $\lambda$ is called an $H$-eigenvalue of $\mathcal{A}$, and $x$ is an $H$-eigenvector of $\mathcal{A}$ associated with $\lambda$ \cite{Qi2005,Lim}.
We denote the set of all $H$-eigenvalues of $\mathcal{A}$ as $\sigma(\mathcal{A})$, which is referred to as the $H$-spectrum. The $H$-spectral radius of $\mathcal{A}$ is defined as:
\begin{equation*}
\rho(\mathcal{A})=\max\{|\lambda|; \lambda\in \sigma(\mathcal{A})\}.
\end{equation*}

Consider an $m$-degree homogeneous polynomial form of $n$ variables, denoted as $f(x)$, where $x\in \mathbb{R}^{n}$ and $\mathcal{A}\in \mathbb{R}^{[m\times n]}$ in equation \eqref{1.1}:
\begin{eqnarray}\label{1.1}
f(x)=\mathcal{A}x^{m}=\sum\limits_{i_1,i_2,\cdots,i_m=1}^{n}a_{i_1i_2\cdots i_m}x_{i_1}x_{i_2}\cdots x_{i_m}.
\end{eqnarray}
When $m$ is even, $f(x)$ is referred to as positive semidefinite if
$$f(x)\geq 0~~~~\forall x\in \mathbb{R}^{n}, x\neq 0.$$
In general, for $n\geq 3$ and $m\geq 4$, determining the positive definiteness of $f(x)$ becomes an NP-hard problem in mathematics. However, Qi \cite{Qi2005} observed that $f(x)$ defined by \eqref{1.1} is positive definite if and only if the real symmetric tensor $\mathcal{A}$ is positive semidefinite.

The product of a tensor $\mathcal{A}$ and a diagonal matrix $D$ is defined as follows \cite{SJY2013}: Let $\mathcal{A}=(a_{i_1i_2\cdots i_m}) \in \mathbb{C}^{m \times n}$ and  $D=\text{diag}(d_1, d_2, \ldots, d_n)$, one has
\begin{equation*}
  \mathcal{B}=(b_{i_1i_2\cdots i_m})=\mathcal{A}D^{m-1},\quad b_{i_1i_2\cdots i_m}=a_{i_1i_2\cdots i_m}d_{i_2}d_{i_3}\cdots d_{i_m},\quad i_1,i_2,\cdots, i_m\in N.
\end{equation*}
The resulting tensor $\mathcal{B}$ is termed as the product of tensor $\mathcal{A}$ and the matrix $D$.


The tensor $\mathcal{A} = (a_{i_1i_2\cdots i_m}) \in \mathbb{C}^{[m\times n]}$ is termed an $H$-tensor \cite{LI20141} if there exists an entrywise positive vector $x = (x_1, x_2, \ldots, x_n)^{\top} \in \mathbb{R}^{n}$ such that for all $i \in N$,
\begin{equation}\label{def3.1}
   |a_{i\cdots i}|x_{i}^{m-1} > \sum\limits_{\substack{i_2,\ldots,i_m\in N,\\ \delta_{i i_2 \cdots i_m} = 0}} |a_{i i_2\cdots i_m}|x_{i_2}\cdots x_{i_m}.
\end{equation}
If $\mathcal{A}$ is an $H$-tensor, then $0 \notin \sigma(\mathcal{A})$.

$\mathcal{A}=(a_{i_1i_2\cdots i_m}) \in \mathbb{C}^{[m\times n]}$
is $Z$-tensor if all of its off-diagonal entries are non-positive, which is equivalent to write $\mathcal{A}=s\mathcal{I}-\mathcal{B}$, where $s>0$ and $\mathcal{B}$ is a nonnegative tensor. $\mathcal{A}$ is called an $M$-tensor \cite{ZLP2014} if there exist a nonnegative tensor $\mathcal{B}$ and a positive real number $s>\rho(\mathcal{B})$ such that
\begin{equation*}
  \mathcal{A}=s\mathcal{I}-\mathcal{B}.
\end{equation*}


\subsection{Preliminaries for tensor quantum information}

In quantum information theory \cite{quantuminformation}, a pure quantum state $|v\rangle\in \mathbb{C}^n$ is a normal vector, and a mixed state $\rho\in M_n(\mathbb{C})$ is a positive semidefinite (PSD) Hermitian matrix with  $\operatorname{Tr}(\rho)=1$. A mixed state $\rho \in M_m(\mathbb{C}) \otimes M_n(\mathbb{C})$ is called separable if there exist pure states $\{|w_{i}\rangle \}_i\subseteq \mathbb{C}^m$ and $\{|v_{i}\rangle \}_i\subseteq \mathbb{C}^n$ such that
\begin{equation*}
  \rho=\sum\limits_{i}p_{i}|w_{i}\rangle \langle w_{i}| \otimes |v_{i}\rangle \langle v_{i}|,~~ p_{i}\geq 0, ~~\sum p_{i}=1,
\end{equation*}
where $\otimes$ denotes the Kronecker product \cite{horn2012matrix}, and $\langle v_{i}|$ is the dual (row) vector of $|v_{i}\rangle$. In other words, $\rho$ is separable if and only if it can be written in the form
$$\rho=\sum\limits_{i}X_{i}\otimes Y_{i},$$
where $X_{i}\in M_m(\mathbb{C})$ and $Y_{i} \in M_n(\mathbb{C})$ are  positive semidefinite matrices. If the state $\rho$ is not separable, it is called an entangled state \cite{RevModPhysQE}.

Spin-$j$ coherent states $|\alpha\rangle$ is defined for $\alpha=e^{-i \varphi} \cot (\theta / 2)$ with $\theta \in[0, \pi]$ and $\varphi \in[0,2 \pi[$ by
\begin{equation}\label{coherntn}
  |\alpha\rangle=\sum_{l=-j}^{j} \sqrt{\left(\begin{array}{c}
2 j \\
j+l
\end{array}\right)}\left[\sin \frac{\theta}{2}\right]^{j-l}\left[\cos \frac{\theta}{2} e^{-i \varphi}\right]^{j+l}|j, l\rangle
\end{equation}
in the standard angular momentum basis $\{|j, l\rangle:-j \leq$ $l \leq j\}$.
The spin coherent state $|\alpha\rangle$ can be considered as a spin-$j$ pointing in the direction $\mathbf{n}=(\sin \theta \cos \varphi, \sin \theta \sin \varphi, \cos \theta)$.

A spin-$j$ state can be seen as a symmetrized state of $2j$ spins $\frac{1}{2}$, and thus it can be written with $m = 2j$ \cite{2010H}. A spin-$j$ coherent state is the tensor product of $2j$ identical spin-$\frac{1}{2}$ coherent states. Since any spin-$\frac{1}{2}$ pure state is a coherent state, spin-$j$ coherent states are by definition the symmetric separable pure states.

A spin-$j$ state $\rho_{\mathcal{C}}$ is classical \cite{Classicality2008} if and only if it can be decomposed into the sum of spin coherent states with positive weights, that is, there are spin coherent states $|\alpha_i\rangle$ such that
\begin{equation}\label{eqclassical}
  \rho_{\mathcal{C}}=\sum\limits_{i}\omega_{i}|\alpha_{i}\rangle \langle \alpha_{i}|, \quad 0\leq \omega_{i}\leq 1, \sum\limits_{i}\omega_{i}=1.
\end{equation}
Since a coherent spin state is equivalent to a symmetric separable pure state, the entangled symmetric multiqubit state can not be written as a classical state, that is, it can not be written as in \eqref{eqclassical}.
Moreover, The literature on \cite{PhysRevA.94.042343} show that classical states are identified with fully separable symmetric states.

\section{The tensor-generated matrices and $H$-tensors}
\label{sec:main}


In this section, we introduce the concept of tensor-generated matrices. These matrices, derived from tensors, are broadly characterized as the associations between an $m$-th order $n$-dimensional tensor $\mathcal{A}$ and an $n$-dimensional square matrix $A$. This classification streamlines the alignment of elements in tensor $\mathcal{A}$ with their counterparts in matrix $A$, both sharing identical dimensions. We establish that if the tensor-generated matrix $A$ is an $H$-matrix, then the tensor $\mathcal{A}$ must be an $H$-tensor. Moreover, we delve into similar properties shared by the original tensor and its derived matrix, including weak irreducibility and weakly chained diagonal dominance. These methodologies facilitate the extension of nearly all properties of pseudo-diagonal dominance from matrices to higher-order tensors, challenging the prior notion that such properties pertained exclusively to matrices rather than tensors.

For the sake of convenience, the following notation will be used throughout the remainder of this paper:
\begin{equation*}
 s_{ij}(\mathcal{A}) =\frac{1}{m-1}\sum\limits_{k=2}^{m}\sum\limits_{\substack{i_t \in N, t\in N\setminus \{1,k\}, i_k=j,\\ \delta_{i i_2 \cdots i_m} =0}}|a_{i i_2\cdots i_m}|, \quad j=1,2,\cdots,n.
\end{equation*}

It is important to note that the value of $m$ is greater than or equal to 2. Specifically, when $m=2$, $s_{ij}(A)$ corresponds to the sum of the elements in a row of the matrix minus the diagonal elements.

\subsection{The tensor-generated matrices and $H$-tensors}

Firstly, we introduce the concept of tensor-generated matrices that characterize as the associations between an $m$-th order $n$-dimensional tensor $\mathcal{A}$ and an $n$-dimensional square matrix $A$.
\begin{definition}\label{TGM}
Consider $\mathcal{A} \in \mathbb{C}^{[m\times n]}$ as an $m$th-order $n$-dimensional tensor. From this tensor, we derive a square matrix $A=(a_{ij}) \in \mathbb{C}^{n, n}$ with $n$ dimensions, defined as:
\begin{equation}\label{matrix A}
    a_{ij}=
\begin{cases}
|a_{ii\cdots i}|-s_{ii}(\mathcal{A}),& i=j,\\
s_{ij}(\mathcal{A}),& i\neq j.
\end{cases}
\end{equation}
This square matrix $A \in \mathbb{C}^{n, n}$ is termed the tensor-generated matrix of tensor $\mathcal{A}$.
\end{definition}

In this context, we specifically map the elements of an $m$th-order $n$-dimensional tensor $\mathcal{A}$ onto a square matrix $A$ containing $n$ dimensions. This mapping involves categorizing the $n^{m}$ entries of $\mathcal{A}$ to establish the correspondences between an $m$th-order $n$-dimensional tensor $\mathcal{A}$ and an $n$-dimensional square matrix $A$.
Subsequently, we show that if the tensor-generated matrix $A \in \mathbb{C}^{n, n}$ is an $H$-matrix, then the tensor $\mathcal{A}$ is also an $H$-tensor.

\begin{theorem}\label{HH}
Let $A\in \mathbb{C}^{n,n}$ be a tensor-generated matrix by tensor $\mathcal{A} \in \mathbb{C}^{[m\times n]}$. If matrix $A$ is an $H$-matrix, then tensor $\mathcal{A}$ is an $H$-tensor.
\end{theorem}
\begin{proof}
It is a known fact that when the tensor $\mathcal{A} \in \mathbb{C}^{[m\times n]}$ qualifies as an $H$-tensor, it satisfies the inequality defined in \eqref{def3.1},
\begin{equation*}
   |a_{i\cdots i}|x_{i}^{m-1}>\sum\limits_{\substack{i_2,\cdots,i_m\in N,\\ \delta_{i i_2 \cdots i_m} =0}}|a_{i i_2\cdots i_m}|x_{i_2}\cdots x_{i_m}.
\end{equation*}

Utilizing the arithmetic-geometric mean inequality, we can represent the right-hand side of the aforementioned inequality as follows:
\begin{equation}\label{eq3.1}
\begin{split}
& \sum\limits_{\substack{i_2,\cdots,i_m\in N,\\ \delta_{i i_2 \cdots i_m} =0}}|a_{i i_2\cdots i_m}|x_{i_2}\cdots x_{i_m}\\
=& \sum\limits_{\substack{i_2,\cdots,i_m\in N,\\ \delta_{i i_2 \cdots i_m} =0}}|a_{i i_2\cdots i_m}|^{\frac{1}{m-1}}x_{i_2}\cdots |a_{i i_2\cdots i_m}|^{\frac{1}{m-1}}x_{i_m}\\
\leq & \frac{1}{m-1}\sum\limits_{\substack{i_2,\cdots,i_m\in N,\\ \delta_{i i_2 \cdots i_m} =0}}(|a_{i i_2\cdots i_m}|x_{i_2}^{m-1}+\cdots+|a_{i i_2\cdots i_m}|x_{i_m}^{m-1})\\
= & \frac{1}{m-1}(\sum\limits_{\substack{i_2,\cdots,i_m\in N,\\ \delta_{i i_2 \cdots i_m} =0}}|a_{i i_2\cdots i_m}|x_{i_2}^{m-1}+\cdots+\sum\limits_{\substack{i_2,\cdots,i_m\in N,\\ \delta_{i i_2 \cdots i_m} =0}}|a_{i i_2\cdots i_m}|x_{i_m}^{m-1})\\
=& \frac{1}{m-1}\sum\limits_{t=2}^{m}\sum\limits_{\substack{i_2,\cdots,i_m\in N,\\ \delta_{i i_2 \cdots i_m} =0}}|a_{i i_2\cdots i_m}|x^{m-1}_{i_t}.
\end{split}
\end{equation}


It is note worthy that the following $m-1$ equalities hold for $p=2,\cdots,m$
\begin{equation}\label{eq3.20}
\begin{split}
  \sum\limits_{\substack{i_2,\cdots,i_m\in N,\\ \delta_{i i_2 \cdots i_m} =0}} & |a_{i i_2\cdots i_m}|x_{i_p}^{m-1}=\sum\limits_{\substack{i_k\in N, i_p=1,k=2\cdots m,\\ \delta_{i i_2 \cdots i_m} =0}}|a_{i i_2\cdots i_m}|x_{1}^{m-1}\\
  &+\cdots+\sum\limits_{\substack{i_k\in N,  i_p=n,k=2\cdots m,\\ \delta_{i i_2 \cdots i_m} =0}}|a_{i i_2\cdots i_m}|x_{n}^{m-1},
\end{split}
\end{equation}
If we set $p=2$, then for $m=2$, we obtain a matrix where the indices $i_3 \ldots i_m$ are eliminated. In this scenario, there is
$$\sum\limits_{\delta_{i i_2} =0}|a_{i i_2}|x_{i_2}=\sum\limits_{\delta_{i i_2 } =0}|a_{i 1}|x_1+\cdots+ \sum\limits_{\delta_{i i_2 } =0}|a_{i n}|x_n.$$

We sum up the $q$-th term from all the aforementioned equalities in \eqref{eq3.20} for $q=1$ to $n$,
\begin{equation*}
\begin{split}
   \sum\limits_{\substack{i_k\in N, i_2=q,k=2\cdots m,\\ \delta_{i i_2 \cdots i_m} =0}} & |a_{i i_2\cdots i_m}|x_{q}^{m-1}+\cdots+\sum\limits_{\substack{i_k\in N,  i_m=q,k=2\cdots m ,\\ \delta_{i i_2 \cdots i_m} =0}}|a_{i i_2\cdots i_m}|x_{q}^{m-1}\\
  =& \sum\limits_{k=2}^{n}\sum\limits_{\substack{i_t \in N, t\in N\setminus \{1,k\}, i_k=q,\\ \delta_{i i_2 \cdots i_m} =0}}|a_{i i_2\cdots i_m}|x_{q}^{m-1}.
\end{split}
\end{equation*}

It follows that
\begin{equation}\label{beq1}
  \sum\limits_{\substack{i_2,\cdots,i_m\in N,\\ \delta_{i i_2 \cdots i_m} =0}}\sum\limits_{t=2}^{m}|a_{i i_2\cdots i_m}|x^{m-1}_{i_t}=\sum\limits_{q=1}^{n}
  \sum\limits_{k=2}^{m}\sum\limits_{\substack{i_t \in N, t\in N\setminus \{1,k\}, i_k=q,\\ \delta_{i i_2 \cdots i_m} =0}}|a_{i i_2\cdots i_m}|x_{q}^{m-1}.
\end{equation}
Based on equations \eqref{eq3.1} and \eqref{beq1}, it can be inferred that
\begin{equation}\label{eq3.90}
  \sum\limits_{\substack{i_2,\cdots,i_m\in N,\\ \delta_{i i_2 \cdots i_m} =0}}|a_{i i_2\cdots i_m}|x_{i_2}\cdots x_{i_m}\leq
  \frac{1}{m-1}\sum\limits_{t=2}^{m}\sum\limits_{\substack{i_2,\cdots,i_m\in N,\\ \delta_{i i_2 \cdots i_m} =0}}|a_{i i_2\cdots i_m}|x^{m-1}_{i_t}=
  \sum\limits_{j=1}^{n}s_{ij}(\mathcal{A})x_{j}^{m-1}.
\end{equation}

In addition, we construct the matrix $B=(b_{ij})=AD\in \mathbb{C}^{n, n}$ as follows;
\begin{equation}\label{matrix A1}
    b_{ij}=
\begin{cases}
(|a_{ii\cdots i}|-s_{ii}(\mathcal{A}))x_{i}^{m-1},& i=j,\\
s_{ij}(\mathcal{A})x_{j}^{m-1},& i\neq j,
\end{cases}
\end{equation}
where $D=diag(x^{m-1}_{1},x^{m-1}_{2},\cdots,x^{m-1}_{n})$.

Given that the tensor-generated matrix $A$ is an $H$-matrix, there exists a positive diagonal matrix $D = \text{diag}(x^{m-1}_{1}, x^{m-1}_{2}, \ldots, x^{m-1}_{n})$ such that $B = AD$ forms a strictly diagonally dominant matrix. Hence, the following inequality holds:
\begin{equation*}
  (|a_{ii\cdots i}|-s_{ii}(\mathcal{A}))x_{i}^{m-1}>\sum\limits_{j\neq i}^{n}s_{ij}(\mathcal{A})x_{j}^{m-1}.
\end{equation*}
Subsequently, the following inequalities are derived using \eqref{eq3.90}:
\begin{equation*}
\begin{split}
  |a_{ii\cdots i}|x_{i}^{m-1}>&\sum\limits_{j=1}^{n}s_{ij}(\mathcal{A})x_{j}^{m-1}\\
  =& \frac{1}{m-1}\sum\limits_{\substack{i_2,\cdots,i_m\in N,\\ \delta_{i i_2 \cdots i_m} =0}}\sum\limits_{t=2}^{m}|a_{i i_2\cdots i_m}|x^{m-1}_{i_t}\\
  \geq & \sum\limits_{\substack{i_2,\cdots,i_m\in N,\\ \delta_{i i_2 \cdots i_m} =0}}|a_{i i_2\cdots i_m}|x_{i_2}\cdots x_{i_m}.
\end{split}
\end{equation*}
Consequently,  tensor $\mathcal{A}$ qualifies as an $H$-tensor. Thus, the proof is concluded.
\end{proof}

\begin{remark}
It is important to note that when $x_{i_2} = \cdots = x_{i_m} = 1$, i.e., when the equality in \eqref{eq3.1} is satisfied, there exist
\begin{equation}\label{beq2}
\begin{split}
r_{i}(\mathcal{A})=& \sum\limits_{\substack{i_2,\cdots,i_m\in N,\\ \delta_{i i_2 \cdots i_m} =0}}|a_{i i_2\cdots i_m}|=
    \frac{1}{m-1}\sum\limits_{\substack{i_2,\cdots,i_m\in N,\\ \delta_{i i_2 \cdots i_m} =0}}\sum\limits_{t=2}^{m}|a_{i i_2\cdots i_m}|\\
    = & \frac{1}{m-1}\sum\limits_{j=1}^{n}
  \sum\limits_{k=2}^{m}\sum\limits_{\substack{i_t \in N, t\in N\setminus \{1,k\}, i_k=j,\\ \delta_{i i_2 \cdots i_m} =0}}|a_{i i_2\cdots i_m}|\\
  =& s_{ii}(\mathcal{A})+\sum\limits_{j=1,j\neq i}^{n}s_{ij}(\mathcal{A}).
\end{split}
\end{equation}
This demonstrates that Definition \ref{TGM} extends the concept of $r_{i}(\mathcal{A})$, representing a particular case within our framework.
\end{remark}



\begin{definition}{\rm\cite[Definition 3.14]{ZLP2014}}\label{defd2}
A tensor $\mathcal{A}=(a_{i_1i_2\cdots i_m}) \in \mathbb{C}^{[m\times n]}$ is called a (strictly) diagonally dominant tensor if for $i\in N$,
\begin{equation}\label{eq:2.2}
 |a_{i\cdots i}|\geq (>) r_{i}(\mathcal{A})=\sum\limits_{\substack{i_2,\cdots,i_m\in N,\\ \delta_{i i_2 \cdots i_m} =0}}|a_{i i_2\cdots i_m}|.
\end{equation}
And we denote
\begin{equation*}
  J(\mathcal{A})=\{i: 1\leq i\leq n: i~~ \text{satisfies strict inequalities in \eqref{eq:2.2}}\}.
\end{equation*}
\end{definition}

Given the construction of the tensor-generated matrix $A\in \mathbb{C}^{n,n}$ from the tensor $\mathcal{A}\in \mathbb{C}^{[m\times n]}$, the following conclusion can be inferred.
\begin{theorem}\label{thedd}
Let $A\in \mathbb{C}^{n,n}$ be a tensor-generated matrix by tensor $\mathcal{A} \in \mathbb{C}^{[m\times n]}$. If the tensor-generated matrix $A$ is (strictly) diagonally dominant matrix, then  tensor $\mathcal{A}$ is (strictly) diagonally dominant tensor. 
\end{theorem}
\begin{proof}
If a tensor-generated matrix $A \in \mathbb{C}^{n,n}$ is derived from tensor $\mathcal{A} \in \mathbb{C}^{[m\times n]}$ and exhibits (strict) diagonal dominance, then the inequality
\begin{equation*}
  |a_{ii\cdots i}| - s_{ii}(\mathcal{A}) \geq (>) \sum\limits_{j=1,j\neq i}^{n} s_{ij}(\mathcal{A})
\end{equation*}
leads to the conclusion derived from \eqref{beq2}:
\begin{equation*}
  |a_{ii\cdots i}| \geq (>) \sum\limits_{j=1}^{n} s_{ij}(\mathcal{A}) = r_{i}(\mathcal{A}).
\end{equation*}
Thus, the proof is concluded.
\end{proof}

\begin{remark}
If there exists some $i_0\in N$ such that $|a_{i_0 i_0 \cdots i_0}|-s_{i_0 i_0}(\mathcal{A})\leq 0$, it is evident that $\mathcal{A}$ is not a diagonally dominant tensor. Therefore, without special explanation, we always discuss diagonal dominance of tensors in the case of $|a_{ii\cdots i}|>s_{ii}(\mathcal{A})$ throughout this paper.
\end{remark}

\begin{remark}
As stated in equation \eqref{beq2}, a tensor satisfying diagonal dominance, as defined in Definition \ref{defd2}, holds true when $|a_{ii\cdots i}| > s_{ii}(\mathcal{A})$.
\end{remark}

According to Theorems \ref{HH} and \ref{thedd}, the study of the class of $H$-tensors can be approached by considering the tensor-generated matrix as a diagonally dominant matrix or an $H$-matrix.

Given an $n$-dimensional the tensor-generated matrix $A = (a_{ij}) \in \mathbb{C}^{n \times n}$, let us define:
\begin{equation*}
  P_{i}(A) = \sum\limits_{j=1,j\neq i}^{n} s_{ij}(\mathcal{A}) = \frac{1}{m-1} \sum\limits_{j\neq i}^{n} \sum\limits_{k=2}^{m} \sum\limits_{\substack{i_t \in N, t\in N\setminus \{1,k\}, i_k=j,\\ \delta_{i i_2 \cdots i_m} = 0}} |a_{i i_2\cdots i_m}|,
\end{equation*}
\begin{equation*}
  Q_{i}(A) = \sum\limits_{i=1,i\neq j}^{n} s_{ij}(\mathcal{A}) = \frac{1}{m-1} \sum\limits_{i\neq j}^{n} \sum\limits_{k=2}^{m} \sum\limits_{\substack{i_t \in N, t\in N\setminus \{1,k\}, i_k=j,\\ \delta_{i i_2 \cdots i_m} = 0}} |a_{i i_2\cdots i_m}|.
\end{equation*}

Various forms of diagonal dominance exist for matrices. Here, we exemplify the fundamental types of diagonally dominant matrices.

For a matrix $A = (a_{ij}) \in \mathbb{C}^{n \times n}$, it is defined as diagonally dominant ($D_n$) if, for each $i \in N$,
\begin{equation}\label{defd}
  |a_{ii}| \geq P_{i}(A).
\end{equation}

Matrix $A = (a_{ij}) \in \mathbb{C}^{n \times n}$ is termed a doubly diagonally dominant matrix ($DD_n$) if, for all $i, j \in N$, $i \neq j$,
\begin{equation}\label{defdd}
  |a_{ii}||a_{jj}| \geq \big(P_{i}(A)\big)\big(P_{j}(A)\big).
\end{equation}

A matrix $A = (a_{ij}) \in \mathbb{C}^{n \times n}$ is considered a $\gamma$-diagonally dominant matrix ($D^{\gamma}_n$) if there exists $\gamma \in [0,1]$ such that
\begin{equation}\label{defo1}
  |a_{ii}| \geq \gamma \big(P_{i}(A)\big) + (1-\gamma)\big(Q_{i}(A)\big), \quad \forall i \in N.
\end{equation}

Matrix $A = (a_{ij}) \in \mathbb{C}^{n \times n}$ is termed a product $\gamma$-diagonally dominant matrix ($PD^{\gamma}_n$) if there exists $\gamma \in [0,1]$ such that
\begin{equation}\label{defo2}
  |a_{ii}| \geq \big(P_{i}(A)\big)^{\gamma} \big(Q_{i}(A)\big)^{1-\gamma}, \quad \forall i \in N.
\end{equation}

When all inequalities in \eqref{defd}-\eqref{defo2} hold strictly, $A$ is categorized as a strictly diagonally dominant matrix ($SD_n$), strictly doubly diagonally dominant matrix ($SDD_n$), strictly $\gamma$-diagonally dominant matrix ($SD^{\gamma}_n$), and strictly product $\gamma$-diagonally dominant matrix ($SPD^{\gamma}_n$), respectively. Furthermore, $A$ is referred to as a generalized diagonally dominant matrix ($GSD_n$), generalized $\gamma$-diagonally dominant matrix ($GSD^{\gamma}_n$), and generalized product $\gamma$-diagonally dominant matrix ($GSPD^{\gamma}_n$) if there exists a positive diagonal matrix $D$ such that $AD$ is a strictly diagonally dominant matrix, strictly $\gamma$-diagonally dominant matrix, and strictly product $\gamma$-diagonally dominant matrix, respectively.

It is well known that a matrix $A$ is a nonsingular $H$-matrix if there exists a positive diagonal matrix $D = (x_1, \cdots, x_n)^{\top}$ ($x_i > 0$) for $i \in N$ such that $AD$ is a strictly diagonally dominant matrix ($SD_n$) \cite{Varga}.

\begin{lemma}{\rm\cite{horn2012matrix,Varga}}\label{Hmat}
Let $A\in \mathbb{C}^{n,n}$. 
If matrix $A$ meets one of the following conditions, then matrix $A$ is a $H$-matrix.

(1) $A$ is a strictly diagonally dominant matrix ($SD_{n}$).

(2) $A$ is strictly doubly diagonally dominant matrix ($SDD_{n}$).

(3) $A$ is a strictly $\gamma$-diagonally dominant matrix ($ SD^{\gamma}_{n}$).

(4) $A$ is a strictly product $\gamma$-diagonally dominant matrix ($ SPD^{\gamma}_{n}$).

(5) $A$ is a generalized diagonally dominant matrix ($GSD_{n}$), generalized $\gamma$-diagonally dominant matrix ($GSD^{\gamma}_{n}$), and generalized product $\gamma$-diagonally dominant matrix ($GSPD^{\gamma}_{n}$).
\end{lemma}

Nonetheless, properties beyond diagonal dominance do not straightforwardly extend to higher-order tensors due to their exclusive applicability to matrices.

However, a connection can be established between these diagonally dominant properties and higher-order tensors through the tensor-generated matrix, facilitating the alignment of $H$-tensor and $H$-matrix evaluations, thereby enhancing $H$-tensor theory.

Likewise, we can illustrate that the subsequent classical matrix outcome extends to higher-order tensors through the tensor-generated matrix in conjunction with Lemma \ref{Hmat}.

\begin{theorem}\label{thm03h}
Consider a tensor $\mathcal{A} \in \mathbb{C}^{[m\times n]}$ that generates a matrix $A\in \mathbb{C}^{n,n}$. If the tensor-generated matrix $A$ satisfies any of the conditions outlined in Lemma \ref{Hmat}, it implies that the tensor $\mathcal{A}$ is an $H$-tensor.
\end{theorem}

It is crucial to highlight that extending fundamental diagonally dominant properties to higher-order tensors via the tensor-generated matrix is feasible.
Theorem \ref{thm03h} illustrates the extension of the remaining properties in Lemma \ref{Hmat} to higher-order tensors, resembling property (1) in Lemma \ref{Hmat}. This establishes the theoretical foundation for the content in the fifth section.
\subsection{The tensor-generated matrix and weakly irreducible diagonally dominant tensor}

Emphasizing that it is feasible to extend fundamental diagonally dominant properties to higher-order tensors via the tensor-generated matrix is crucial. Most of the determination findings for $H$-matrices can be extended to higher-order tensors, including weakly irreducible diagonally dominant tensors and weakly chained diagonally dominant tensors.

\begin{definition}{\rm\cite[Definition 2.1]{KCC}}
Let $\mathcal{A}=(a_{i_1i_2\cdots i_m}) \in \mathbb{C}^{[m\times n]}$. Then $\mathcal{A}$ is called reducible, if there exists a nonempty proper index $I\subset N$ such that
\begin{equation*}
  a_{i_1i_2\cdots i_m}=0,\quad \forall i_1\in I, \quad \forall i_2, \cdots ,i_m\notin I.
\end{equation*}
If $\mathcal{A}$ is not reducible, then $\mathcal{A}$ is called irreducible.
\end{definition}

\begin{definition}{\rm\cite[Definition 2.2]{HSL2014}}
Let $\mathcal{A}=(a_{i_1i_2\cdots i_m}) \in \mathbb{C}^{[m\times n]}$ and $R(|\mathcal{A}|)$, the representation of $\mathcal{A}$, be a $n\times n$ matrix whose $(i,j)$-th entry is given by
\begin{equation*}
  R(|\mathcal{A}|)_{ij}=\sum\limits_{j\in \{i_2\cdots i_m\} }|a_{i i_2\cdots i_m}|.
\end{equation*}
Tensor $\mathcal{A}$ is called weakly reducible if $R(|\mathcal{A}|)$ is a reducible matrix. Otherwise, tensor $\mathcal{A}$ is weakly irreducible
\end{definition}

There is another equivalent about weakly irreducible tensor.
An order $m$ dimension $n$ tensor $\mathcal{A}$ is called weakly reducible, if there exists a nonempty proper index subset $I\subset N$ such that
\begin{equation*}
   a_{i_1i_2\cdots i_m}=0,\quad \forall i_1\in I, \quad \exists i_j\notin I, \quad j=2,\cdots,m.
\end{equation*}
If $\mathcal{A}$ is not weakly reducible, then we call $\mathcal{A}$ weakly irreducible \cite{HSL2014,MNA2015}.
In the context of matrices, it is worth noting that there is no distinction between irreducible and weakly irreducible matrices. However, while an irreducible tensor must also be weakly irreducible, the reverse is not necessarily true.

\begin{lemma}{\rm\cite[Theorem 3.4]{MNA2015}}\label{irh}
Let $\mathcal{A}=(a_{i_1i_2\cdots i_m}) \in \mathbb{C}^{[m\times n]}$. If $\mathcal{A}$ is irreducible,
\begin{equation*}
  |a_{i\cdots i}|\geq r_{i}(\mathcal{A}), \quad \forall i\in N,
\end{equation*}
and strictly inequality holds for at least one $i$, then $\mathcal{A}$ is an $H$-tensor.
\end{lemma}

It can be demonstrated that the tensor-generated matrix $A$ and tensor $\mathcal{A}$ exhibit equivalent weak irreducibility, facilitating the extension of conclusions concerning $H$-matrices from irreducibility to higher-order tensors.

\begin{theorem}\label{thmdir}
Consider $A\in \mathbb{C}^{n,n}$ as the matrix generated by the tensor $\mathcal{A} \in \mathbb{C}^{[m\times n]}$. Then the tensor $\mathcal{A}$ is weakly irreducible if and only if the tensor-generated matrix $A$ is irreducible.
\end{theorem}
\begin{proof}
Necessity: Assume that $\mathcal{A}$ is a weakly irreducible tensor. Suppose the tensor-generated matrix $A$ is reducible. Then there exists a nonempty proper index subset $I\subset N$ such that
$a_{ij}=0, \forall i\in I,  \forall j\notin I$, meaning that we have
\begin{equation*}
 a_{ij}=s_{ij}(\mathcal{A}) =\frac{1}{m-1}\sum\limits_{k=2}^{m}\sum\limits_{\substack{i_2,\cdots,i_{k-1},i_{k+1},\cdots,i_m \in N, i_k=j,\\ \delta_{i i_2 \cdots i_m} =0}}|a_{i i_2\cdots i_m}|=0, j\neq i.
\end{equation*}
Additionally, we find $a_{i i_2\cdots i_m}=0, \forall i\in I,  \exists i_2 \cdots i_m\notin I$,
implying that $\mathcal{A}$ is a weakly reducible tensor, leading to a contradiction.

Sufficiency: Given that the tensor-generated matrix $A$ is irreducible. Assuming that the tensor $\mathcal{A}$ is weakly reducible, there exists a nonempty proper index subset $I\subset N$ such that $a_{i i_2\cdots i_m}=0, \forall i\in I,  \exists i_2 \cdots i_m\notin I$, which implies
\begin{equation*}
   s_{ij}(\mathcal{A}) =\frac{1}{m-1}\sum\limits_{k=2}^{m}\sum\limits_{\substack{i_2,\cdots,i_{k-1},i_{k+1},\cdots,i_m \in N, i_k=j,\\ \delta_{i i_2 \cdots i_m} =0}}|a_{i i_2\cdots i_m}|=0, j\neq i.
\end{equation*}
Equivalently, there exists a nonempty proper index subset $I\subset N$ such that $a_{ij}= s_{ij}(\mathcal{A})=0, \forall i\in I,  \forall j\notin I$, contradicting the irreducibility of matrix $A$. This concludes the proof.
\end{proof}

\begin{corollary}\label{cor:3.5}
Consider $A \in \mathbb{C}^{n,n}$ as a matrix generated by the tensor $\mathcal{A} \in \mathbb{C}^{[m\times n]}$. If the tensor $\mathcal{A}$ is irreducible, then the tensor-generated matrix $A$ is also irreducible.
\end{corollary}

Leveraging Theorem \ref{thedd}, Theorem \ref{thmdir}, and Lemma \ref{irh}, we can draw the subsequent deductions.

\begin{theorem}\label{thm3.66}
Let $A\in \mathbb{C}^{n,n}$ be a tensor-generated matrix by tensor $\mathcal{A} \in \mathbb{C}^{[m\times n]}$. If the tensor-generated matrix $A$ is irreducible diagonally dominant matrix and strictly inequality holds for at least one, then $\mathcal{A}$ is weakly irreducible diagonally dominant tensor. Moreover, tensor $\mathcal{A}$ is $H$-tensor.
\end{theorem}
\begin{proof}
If the tensor-generated matrix $A\in \mathbb{C}^{n,n}$ by tensor $\mathcal{A} \in \mathbb{C}^{[m\times n]}$ is irreducible, then Theorem \ref{thmdir} implies that tensor $A$ is weakly irreducible.
Since the tensor-generated matrix $A\in \mathbb{C}^{n,n}$ is diagonally dominant, according to Theorem \ref{thedd}, tensor $\mathcal{A}$ is diagonally dominant. Furthermore, based on Lemma \ref{irh}, tensor $\mathcal{A}$ is an $H$-tensor. The proof is completed.
\end{proof}

\begin{remark}
Since tensor $\mathcal{A}$ is weakly irreducible if and only if tensor-generated matrix $A\in \mathbb{C}^{n,n}$ by tensor $\mathcal{A}$ is irreducible. This aligns with the finding from (Theorem 3.4 in \cite{MNA2015}) that a weakly irreducible diagonally dominant tensor is an $H$-tensor.
\end{remark}

\begin{corollary}
Consider $A\in \mathbb{C}^{n,n}$ as a tensor-generated matrix by irreducible tensor $\mathcal{A} \in \mathbb{C}^{[m\times n]}$. If the tensor-generated matrix $A$ is diagonally dominant matrix and strictly inequality holds for at least one, then  $\mathcal{A}$ is $H$-tensor.
\end{corollary}

Theorem \ref{thmdir} demonstrates that both the tensor and its tensor-generated matrix demonstrate equivalent weak irreducibility. This observation directs our focus towards exploring the properties of $H$-tensors in the realm of weak irreducibility.

\begin{lemma}{\rm\cite{horn2012matrix,Varga}}\label{le3.8}
If matrix $A\in \mathbb{C}^{n,n}$ meets one of the following conditions, then matrix $A$ is a $H$-matrix.

(1) $A$ is a  irreducible diagonally dominant matrix and strictly inequality holds for at least one.

(2) $A$ is a irreducible doubly diagonally dominant matrix and strictly inequality holds for at least one.

(3) $A$ is a  irreducible $\gamma$-diagonally dominant matrix and strictly inequality holds for at least one.

(4) $A$ is a  irreducible product $\gamma$-diagonally dominant matrix and strictly inequality holds for at least one.
\end{lemma}

Theorem \ref{thm3.66} affirms the feasibility of extending from case (1) in Lemma \ref{le3.8} to higher-order tensors. This extension is not limited to that particular case alone; indeed, all scenarios outlined in Lemma \ref{le3.8} are generalizable to the tensor scenario through the utilization of the tensor-generated matrix.

In essence, the deductions stemming from Lemma \ref{le3.8} and Theorem \ref{HH} lead us to the following conclusions. The proof methodology mirrors a comparable approach, and we will avoid redundancy by omitting its repetition.

\begin{theorem}
Consider $A\in \mathbb{C}^{n,n}$ as a tensor-generated matrix by tensor $\mathcal{A} \in \mathbb{C}^{[m\times n]}$. If matrix $A$ satisfies one of the conditions in  Lemma \ref{le3.8}, then  $\mathcal{A}$ is an $H$-tensor.
\end{theorem}

%

\subsection{The tensor-generated matrix and weakly chained diagonally dominant tensor}

It is crucial to emphasize the possibility of extending fundamental diagonally dominant properties to higher-order tensors through the utilization of the tensor-generated matrix. Most determinant findings for $H$-matrices are transferable to higher-order tensors, encompassing weakly chained diagonally dominant tensors. Prior to outlining our conclusions, it is necessary to introduce pertinent findings within the directed graph linked to tensors.

\begin{definition}{\rm\cite[Definition 13]{Parsiad1}}
Let $\mathcal{A}=(a_{i_1i_2\cdots i_m}) \in \mathbb{C}^{[m\times n]}$ be a tensor.
\rm{(i)} The directed graph of $\mathcal{A}$, denoted graph$\mathcal{A}$, is a tuple $(V,E)$ consisting of the vertex set $V=\{1,\cdots,n\}$ and edge set $E\subset V\times V$ satisfying $(i,j)\in E$ if and only if
$a_{i_1i_2\cdots i_m}\neq 0$ for some $(i_2,\cdots, i_m)$ such that $j\in \{i_2,\cdots, i_m\}$.
\rm{(ii)} A walk in graph$\mathcal{A}$ is a nonempty finite sequence $(i^{1},i^{2}),(i^{2},i^{3}),\cdots, (i^{k-1},i^{k})$ of ``adjacent" edges in $E$.
\end{definition}

\begin{lemma}{\rm\cite[Lemma 14]{Parsiad1}}\label{lem:3.8}
graph$\mathcal{A}$=graph$R(|\mathcal{A}|)$ for any tensor $\mathcal{A}$.
\end{lemma}

\begin{definition}{\rm\cite[Definition 15]{Parsiad1}}\label{def:3.7}
A tensor $\mathcal{A}$ is weakly chained diagonally dominant if all of the following are satisfied:
\rm{(i)} $\mathcal{A}$ is diagonally dominant.
\rm{(ii)} $J(\mathcal{A})$ is nonempty.
\rm{(iii)} For each $i^{1}\notin J(\mathcal{A})$, there exists a walk $i^{1}\rightarrow i^{2}\rightarrow \cdots \rightarrow i^{k}$ in graph $\mathcal{A}$ such that $i^{k}\in J(\mathcal{A})$.
\end{definition}

\begin{lemma}{\rm\cite[Theorem 1]{Parsiad1}}\label{lem:3.9}
Let $\mathcal{A}$ be a weakly chained diagonally dominant $Z$-tensor with nonnegative diagonals, then $\mathcal{A}$ is a $M$-tensor.
\end{lemma}


It can be demonstrated that when tensor $\mathcal{A}$ exhibits weakly chained diagonally dominant characteristics, the corresponding tensor-generated matrix $A$ is diagonally dominant. Moreover, a non-zero element chain exists from $i \notin J(\mathcal{A})$ to $j \in J(\mathcal{A})$ for $i, j \in N$ (where $i \neq j$).

\begin{theorem}
Let $A \in \mathbb{C}^{n,n}$ be a tensor-generated matrix by tensor $\mathcal{A} \in \mathbb{C}^{[m \times n]}$. If the tensor-generated matrix $A$ is a diagonally dominant matrix with a non-zero element chain from $i \notin J(\mathcal{A})$ to $j \in J(\mathcal{A})$ for $i, j \in N$, then tensor $\mathcal{A}$ is a weakly chained diagonally dominant tensor. Furthermore, tensor $\mathcal{A}$ is an $H$-tensor.
\end{theorem}

\begin{proof}
Drawing upon Theorem \ref{thedd}, we promptly deduce (i) and (ii) from Definition \ref{def:3.7}. Furthermore, according to Lemma \ref{lem:3.8}, if the tensor-generated matrix $A$ has a non-zero element chain from $i \notin J(A)$ to $j \in J(A)$ for $i, j \in N$, then $\mathcal{A}$ also features a path $i^{1} \rightarrow i^{2} \rightarrow \cdots \rightarrow i^{k}$ within the graph $\mathcal{A}$ such that $i^{k} \in J(\mathcal{A})$, where $i^{1} = i$ and $i^{k} = j$. Consequently, tensor $\mathcal{A}$ aligns with the definition of weakly chained diagonally dominant tensors as per Definition \ref{def:3.7}. By virtue of Lemma \ref{lem:3.9}, tensor $\mathcal{A}$ qualifies as an $H$-tensor. This concludes the proof.
\end{proof}

In summary, when faced with strict inequality, weak irreducibility, or weak chains, we have the ability to convert the problem of determining an $H$-tensor into the problem of determining an $H$-matrix. This conversion provides us with the advantage of leveraging the well-established $H$-matrix theory to address more intricate problems involving $H$-tensors. As a result, the development of $H$-tensor theory is significantly advanced. 

\section{Application in Classicality for spin states}
\label{sec:quantum}

This section delves into the application of tensor-generated matrices in classical spin states. Leveraging the spectral properties of $H$-tensors from Section 3, we can establish adequate criteria for positive definite tensors. Utilizing the tensor representation of spin states and tensor-generated matrices in symmetric scenarios, we introduce criteria for separability (classicality) in spin-$j$ states.

\subsection{Classicality for spin states with symmetric coefficient tensors}
The literature referenced in \cite{2014T} introduces a tensorial representation for spin states, expanding a spin-$j$ density matrix $\rho$ as follows:
The $4^{m}$ Hermitian matrices $S_{\mu_{1} \ldots \mu_{m}}$ (with $m=2 j$ ) is defined as
$$
S_{\mu_{1} \cdots \mu_{m}}=P\left(\sigma_{\mu_{1}} \otimes \sigma_{\mu_{2}} \cdots  \otimes \sigma_{\mu_{m}}\right) P^{\dagger}, \quad 0 \leqslant \mu_{i} \leqslant 3,
$$
with $P$ being the projector onto the symmetric subspace of tensor products of $N$ spins- $\frac{1}{2}$ (the subspace spanned by Dicke states).
Here, $\sigma_{a}, 1 \leqslant a \leqslant 3$, be the usual Pauli matrices and $\sigma_{0}$ be the $2 \times$ 2 identity matrix as follows:
\begin{equation*}
  \sigma_1 = \begin{pmatrix} 0 & 1 \\ 1 & 0 \end{pmatrix},
\sigma_2 = \begin{pmatrix} 0 & -i \\ i & 0 \end{pmatrix},
\sigma_3 = \begin{pmatrix} 1 & 0 \\ 0 & -1 \end{pmatrix},
\sigma_0 = \begin{pmatrix} 1 & 0 \\ 0 & 1 \end{pmatrix}.
\end{equation*}

For a general spin-$j$, the 
$S_{\mu_{1}\mu_{2}\cdots \mu_{m}}$ provide an overcomplete basis over which $\rho$ can be
expanded, that is, any state can be expressed as
\begin{equation}\label{eqOc2}
  \rho=\frac{1}{2^{m}}\mathcal{A}_{\mu_{1}\mu_{2}\cdots \mu_{N}}S_{\mu_{1}\mu_{2}\cdots \mu_{m}},
\end{equation}
with coefficients
\begin{equation}\label{eqOc3}
  \mathcal{A}_{\mu_{1}\mu_{2}\cdots \mu_{m}}=tr(\rho S_{\mu_{1}\mu_{2}\cdots \mu_{m}}),
\end{equation}
real and invariant under permutation of the indices.
These $\mathcal{A}_{\mu_{1}\mu_{2}\cdots \mu_{m}}$ can be viewed as a symmetric order-$N$ tensor. Thus, we denote Equation \eqref{eqOc3} as the tensor representation of $\rho$.


The tensor corresponding to a spin coherent state $|\alpha\rangle$ aligned with the direction $\textbf{n}$ is represented as
\begin{equation*}
  \mathcal{A}_{\mu_{1}\mu_{2}\cdots \mu_{m}}=\langle\alpha|S_{\mu_{1}\mu_{2}\cdots \mu_{m}}|\alpha\rangle=n_{\mu_{1}}n_{\mu_{2}}\cdots n_{\mu_{m}},
\end{equation*}
where $n_{0}=1$ and $\textbf{n}=(n_{1}, n_{2}, n_{3})$. 

It is worth noting that the generalized Bloch representation \eqref{eqOc2} shares several crucial properties with the Bloch representation of a spin-$\frac{1}{2}$. From \eqref{eqOc3}, we observe that the coordinates of a coherent state are derived simply by multiplying the components of the 4-vector $n_{\mu}^{(i)}=(1,\textbf{n}^{(i)})$. This extends the concept that the Bloch vector representing a spin-$\frac{1}{2}$ state aligns with the directions defined by the angles of the coherent state. Consequently, we have $\mathbf{n}=(\sin \theta \cos \varphi, \sin \theta \sin \varphi, \cos \theta)\geq 0$, where $\theta \in[0, \frac{\pi}{2}]$ and $\varphi \in[0, \frac{\pi}{2}]$.

The concept of classicality defined in \eqref{eqclassical} can be reformulated using tensors.
A state is deemed classical if and only if there exist positive weights $\omega_{i}$ and unit vectors $\textbf{n}^{(i)}$ such that its tensor of coordinates $\mathcal{A}$ can be expressed as
\begin{equation}\label{eqw}
  \mathcal{A}
  =\sum\limits_{i}\omega_{i}n_{\mu_{1}}^{(i)}n_{\mu_{2}}^{(i)}\cdots n_{\mu_{m}}^{(i)}=\sum\limits_{i=1}^{r}\omega_{i}\big(n_{\mu}^{(i)}\big)^{\otimes m},
\end{equation}
where $n_{\mu}^{(i)}=(1,\textbf{n}^{(i)})$. It is evident that $\mathcal{A}$ is a symmetric tensor of rank $r$.


Within Section \ref{sec:main}, when strict inequality, weak irreducibility, or weak chain conditions are present, the calculation of the $H$-tensor can be streamlined to computing the $H$-matrix.
This streamlined approach allows us to apply the well-established $H$-matrix theory to address more intricate $H$-tensor challenges, thereby significantly advancing the realm of $H$-tensor theory. Notably, we can leverage the spectral characteristics of $H$-tensors to establish sufficient conditions for positive semidefinite tensors.

\begin{lemma}\label{Hmatrix-Htensor}
Consider $\mathcal{A} \in \mathbb{C}^{[m\times n]}$ as an even-order real symmetric tensor where $a_{kk\cdots k}\geq 0$ for all $k\in \mathbb{N}$, and let $A\in \mathbb{C}^{n,n}$ represent the tensor-generated matrix by the tensor $\mathcal{A}$. If matrix $A$ qualifies as an $H$-matrix, it implies that the tensor $\mathcal{A}$ is positive semidefinite.
\end{lemma}

In this context, we introduce the criterion for determining the classicality of spin states, drawing from Lemma \ref{Hmatrix-Htensor} and the relationship established in Section \ref{sec:main}.

\begin{theorem}
For a spin-$j$ state, if its representation tensor $\mathcal{A}_{\mu_{1}\cdots \mu_{2j}} \in \mathbb{R}^{[2j\times 4]}$ is a real tensor with $a_{kk\cdots k}\geq 0$ for all $k\in \{1,2,3,4\}$, and $A\in \mathbb{C}^{4,4}$ is a tensor-generated matrix by tensor $\mathcal{A}$. If matrix $A$ is a $H$-matrix, the tensor $\mathcal{A}$ is positive semidefinite. In other words, spin-$j$ state is classical.
\end{theorem}
\begin{proof}
Given a spin-$j$ state $\rho$, there exists a corresponding coefficient tensor $\mathcal{A}_{\mu_{1}\mu_{2}\cdots \mu_{m}}=tr(\rho S_{\mu_{1}\mu_{2}\cdots \mu_{m}})$ as shown in \eqref{eqOc3}. A spin-$j$ coherent state can be seen as a symmetrized tensor product of $2j$ spins $\frac{1}{2}$ coherent states, and thus it can be written with representation tensor $\mathcal{A}_{\mu_{1}\cdots \mu_{2j}} \in \mathbb{R}^{[2j\times 4]}.$

If the representation tensor $\mathcal{A}_{\mu_{1}\cdots \mu_{2j}} \in \mathbb{R}^{[2j\times 4]}$ is a real tensor with $a_{kk\cdots k}\geq 0$ for all $k\in \{1,2,3,4\}$, and the resulting tensor-generated matrix $A\in \mathbb{C}^{4,4}$ is an $H$-matrix, then by Lemma \ref{Hmatrix-Htensor}, we can conclude that the representation tensor $\mathcal{A}_{\mu_{1}\cdots \mu_{2j}} \in \mathbb{R}^{[2j\times 4]}$ is positive semidefinite.

When $j$ is an integer (thus $m=2j$ is even), the right-hand side of equation \eqref{eqw} is consistently positive. For any $x\in \mathbb{R}^{n}$, it follows that
\begin{equation*}
  \mathcal{A}x^{m}=\sum\limits_{i}\omega_{i}n_{\mu_{1}}^{(i)}n_{\mu_{2}}^{(i)}\cdots n_{\mu_{m}}^{(i)}x_{\mu_{1}}x_{\mu_{2}}\cdots x_{\mu_{m}}=\sum\limits_{i=1}^{r}\omega_{i}\left[\left(n_{\mu}^{(i)}\right)^{\top} x\right]^{\otimes m}\geq 0,
\end{equation*}
illustrating the positive semidefiniteness of the even-order tensor $\mathcal{A}$. Consequently, we have demonstrated that the spin-$j$ state adheres to classicality under condition \eqref{eqw}. The proof is completed.
\end{proof}

It is worth noting that when $j$ is an integer, all expansion coefficients correspond to tensors that are fourth-order tensors of even order. This not only establishes a theoretical foundation for deriving the aforementioned criterion using even-order semi-positive definite tensors but also preserves the physical significance of all tensor elements. Specifically, the components of all 4-dimensional vectors $n_{\mu}^{(i)}=(1,\textbf{n}^{(i)})$ equate to 1 and correspond to the polar coordinates of points $\textbf{n}^{(i)}$ on the Bloch sphere.

\subsection{Classicality for spin states with strongly symmetric coefficient tensors}


Quantum tasks are significantly streamlined in symmetric scenarios due to the restriction of optimization to symmetrically separable states. Yet, dealing with the complexity of multipartite states remains a formidable challenge. Addressing these challenges requires exploring quantum states existing within a smaller space compared to symmetric states. A recent contribution by Xiong et al. \cite{xiong2023multipartite} introduces the notion of multipartite strongly symmetric states, representing a subspace within symmetric states.

\begin{definition}{\rm\cite{xiong2023multipartite,stronglysymmetric}}\label{stronglysymmetric}
For $\mathcal{A}=(a_{i_1i_2\cdots i_m}) \in \mathbb{R}^{[m\times n]}$, we denote
\begin{equation*}
  \mathcal{N}=\{(i_1 i_2\cdots i_m): 1\leq i_k \leq n, k=1,2,\cdots,m \}.
\end{equation*}
For $(i_1i_2\cdots i_m)\in \mathcal{N}$, denote by $[(i_1 i_2\cdots i_m)]$ the set of all distinct members in $\{i_1 i_2\cdots i_m\}$.
Let $(i_1 i_2\cdots i_m)\in \mathcal{N}$, $(j_1 j_2\cdots j_m)\in \mathcal{N}$. We say that $(i_1 i_2\cdots i_m)$ is similar to $(j_1 j_2\cdots j_m)$, if $[(i_1 i_2\cdots i_m)]=[(j_1 j_2\cdots j_m)]$, denote by $(i_1 i_2\cdots i_m)\sim (j_1 j_2\cdots j_m)$. We call $\mathcal{A}$ is strongly symmetric, if for any $(i_1 i_2\cdots i_m), (j_1 j_2\cdots j_m)\in \mathcal{N}$,  $(i_1 i_2\cdots i_m)$ $\sim (j_1 j_2\cdots j_m)$, it holds that $a_{i_1i_2\cdots i_m}=a_{j_1j_2\cdots j_m}$.
If $m$-multipartite state $|\Phi\rangle$ is corresponding to a strongly symmetric tensor $\mathcal{A}_{|\Phi\rangle}$.
We denote all $m$-multipartite strongly symmetric state of $n$ qudits as $SS^{m}(\mathbb{R}^{n})$.
\end{definition}

It is important to highlight that a real strongly symmetric tensor is symmetric, with the exception of when $m=2$. Subsequently, we demonstrate that the symmetry observed in the tensor-generated matrix from the tensor aligns with that strong symmetry of the original tensor.

\begin{theorem}\label{doublesymmetric}
Let $A\in \mathbb{C}^{n,n}$ be a tensor-generated matrix by tensor $\mathcal{A} \in \mathbb{C}^{[m\times n]}$.  If the tensor $\mathcal{A}$ is strongly symmetric, if and only if the tensor-generated matrix $A$ is symmetric.
\end{theorem}
\begin{proof}
Necessity: Since $\mathcal{A}$ is a  symmetric tensor. Then it holds that $a_{i_1i_2\cdots i_m}=a_{j_1j_2\cdots j_m}$ for
$(i_1 i_2\cdots i_m)$, $(j_1 j_2\cdots j_m)\in \mathcal{N}$,  $(i_1 i_2\cdots i_m)= \pi(j_1 j_2\cdots j_m)$.
For the tensor-generated matrix $A\in \mathbb{C}^{n,n}$, it follows that
\begin{equation*}
\begin{split}
  s_{ij}(\mathcal{A}) &=\frac{1}{m-1}\sum\limits_{k=2}^{m}\sum\limits_{\substack{i_t \in N, t\in N\setminus \{1,k\}, i_k=j,\\ \delta_{i i_2 \cdots i_m} =0}}|a_{i i_2\cdots i_m}|\\
  &=\frac{1}{m-1}\sum\limits_{k=2}^{m}\sum\limits_{\substack{i_t \in N, t\in N\setminus \{1,k\}, i_k=i,\\ \delta_{j i_2 \cdots i_m} =0}}|a_{j i_2\cdots i_m}|\\
  &=s_{ji}(\mathcal{A}).
\end{split}
\end{equation*}
This shows that the tensor-generated matrix $A$ is symmetric.

Sufficiency: The tensor-generated matrix $A\in \mathbb{C}^{n,n}$ is symmetric, yields
\begin{equation*}
  \sum\limits_{\substack{i_t \in N, t\in N\setminus \{1,k\}, i_k=j,\\ \delta_{i i_2 \cdots i_m} =0}}|a_{i i_2\cdots i_m}|=
  \sum\limits_{\substack{i_t \in N, t\in N\setminus \{1,k\}, i_k=i,\\ \delta_{i i_2 \cdots i_m} =0}}|a_{j i_2\cdots i_m}|.
\end{equation*}
In other word, one has
\begin{equation*}
  |a_{i i_2\cdots i_{k-1} j i_{k+1} \cdots i_m}|=|a_{j i_2\cdots i_{k-1} i i_{k+1}\cdots i_m}|, i_t \in N, t\in N\setminus \{1,k\},
\end{equation*}
which indicates that tensor $\mathcal{A}$ is strong symmetric. The proof is completed.
\end{proof}

Subsequently, we introduce the criterion for determining the classicality of spin states with strongly symmetric coefficient tensors.

\begin{theorem}
For a spin-$j$ state, if its representation tensor $\mathcal{A}_{\mu_{1}\cdots \mu_{2j}} \in \mathbb{R}^{[2j\times 4]}$ is a strongly symmetric tensor, and tensor-generated matrix $A\in \mathbb{C}^{4,4}$ is a $H$-tensor with $|a_{ii\cdots i}|\geq s_{ii}(\mathcal{A})$. In this situation, spin-$j$ state is symmetric classical.
\end{theorem}
\begin{proof}
According to Theorem \ref{doublesymmetric}, the tensor-generated matrix $A\in \mathbb{C}^{4,4}$ exhibits symmetry. If the tensor-generated matrix $A\in \mathbb{C}^{4,4}$ is an $H$-tensor with $|a_{ii\cdots i}|\geq s_{ii}(\mathcal{A})$, then the matrix $A$ possesses nonnegative diagonal elements, indicating its positive semidefiniteness. Drawing on Lemma \ref{Hmatrix-Htensor}, we deduce that the representation tensor $\mathcal{A}_{\mu_{1}\cdots \mu_{2j}} \in \mathbb{R}^{[2j\times 4]}$ is positive semidefinite. Consequently, the spin-$j$ state demonstrates symmetric classical behavior.
\end{proof}

The authors in \cite{stronglysymmetric} demonstrate that a strongly symmetric hierarchically dominated nonnegative tensor is equivalent to a completely positive tensor. Moreover, they specify that a completely positive tensor of even order $m$ is a positive semidefinite tensor, i.e.,
\begin{equation*}
  \mathcal{A} x^{m}=\sum_{k=1}^{r}\left(u^{(k)}\right)^{\otimes m} x^{m}=\sum_{k=1}^{r}\left[\left(u^{(k)}\right)^{\top} x\right]^{\otimes m} \geq 0, \quad \forall x\in \mathbb{R}^{n}.
\end{equation*}

In essence, determining that the coefficient tensor $\mathcal{A}$ of a strongly symmetric order-$m$ is either a strongly symmetric hierarchically dominated nonnegative tensor or a completely positive tensor implies the classical nature of the spin-$j$ state.

\begin{corollary}
For a spin-$j$ state, if its representation tensor $\mathcal{A}_{\mu_{1}\cdots \mu_{2j}} \in \mathbb{R}^{[2j\times 4]}$ is a strongly symmetric hierarchically dominated nonnegative tensor, then tensor-generated matrix $A\in \mathbb{C}^{4,4}$ is a symmetric $H$-tensor and a copositive matrix. This shows spin-$j$ state is symmetric classical, i.e,
\begin{equation*}
  \mathcal{A} =\sum_{k=1}^{r}\left(u^{(k)}\right)^{\otimes m}.
\end{equation*}
\end{corollary}

In this section, we employed the matrix generated by tensors to analyze the classical characteristics of spin-$j$ states. We investigated the characteristics of the coefficient tensor under both symmetric and strongly symmetric conditions. Furthermore, we elucidated a direct correlation between the strong symmetry of the tensor and the symmetry of the tensor-generating matrix.
\section{Application in $H$-Eigenvalue inclusion sets for high order tensors}
\label{sec:inclusion}

%

In this section, the focus is on the inclusion sets of $H$-Eigenvalues for higher-order tensors. Traditional eigenvalue localization results for matrices, including Brauer's Ovals of Cassini sets, Ostrowski sets, and $S$-type sets, do not directly apply to tensor eigenvalues. To address this challenge, an effort is made to establish a connection between eigenvalue localization for higher-order tensors and tensor-generated matrices, expanding on the connection established in Section 3.

By leveraging this connection, the distribution of $H$-eigenvalues for higher-order tensors is determined by linking it to the distribution of eigenvalues for the subset of $H$-matrices. This approach allows for the extension of most matrix eigenvalue localization sets to higher-order tensor eigenvalues. Additionally, significant inclusion sets for tensor eigenvalues are introduced, including modified versions of Brauer's Ovals of Cassini sets, Ostrowski sets, and $S$-type inclusion sets.

\subsection{The applications of tensor-generated in $H$-Eigenvalue inclusion sets for high order tensors}

 Ger$\breve{s}$gorin disc theorem is fundamental and famous in matrix eigenvalue localization. Qi \cite{Qi2005} introduces the $H$-eigenvalue localization set for real symmetric tensors, which extends the Ger$\check{s}$gorin matrix eigenvalue inclusion theorem for matrices.
\begin{theorem}{\rm\cite[Theorem 6]{Qi2005}}\label{thm1.1}
Let $\mathcal{A}=(a_{i_1i_2\cdots i_m}) \in \mathbb{C}^{[m\times n]}$, $n\geq 2$. Then
\begin{equation*}
  \sigma(\mathcal{A})\subseteq \Gamma(\mathcal{A})=\bigcup_{i\in N} \Gamma_{i}(\mathcal{A}),
\end{equation*}
where
$\Gamma_{i}(\mathcal{A})=\{z\in \mathbb{C}:|z-a_{i\cdots i}|\leq r_{i}(\mathcal{A})\},\quad r_{i}(\mathcal{A})=\sum\limits_{\substack{i_2,\cdots, i_m\in N,\\\delta_{ii_2\cdots i_m} =0}}|a_{i i_2\cdots i_m}|.$
\end{theorem}

Building upon the insights from Section 3, the Ger$\breve{s}$gorin disc theorem is extended to the $H$-Eigenvalue of high-order tensors as follows.

\begin{theorem}\label{MG}
Let $\mathcal{A}=(a_{i_1i_2\cdots i_m}) \in \mathbb{C}^{[m\times n]}$, $n\geq 2$. Then
\begin{equation*}
  {\sigma(\mathcal{A})} \subseteq {\mathcal{F}(\mathcal{A})}=\bigcup_{i\in N} {\mathcal{F}_{i}}(\mathcal{A}),
\end{equation*}
where %
 $ {\mathcal{F}_{i}}(\mathcal{A})=\{z\in \mathbb{C}:[|z-a_{i\cdots i}|-s_{ii}(\mathcal{A})]\leq \sum\limits_{j=1,j\neq i}^{n}s_{ij}(\mathcal{A})\}$.
\end{theorem}
\begin{proof}
For any $\lambda \in \sigma(\mathcal{A})$, let $\mathcal{B}=(b_{i_1 i_2\cdots i_m})=\lambda\mathcal{I}-\mathcal{A}$. Then
\begin{equation*}
  0\in \sigma(\mathcal{B}),\quad  r_{i}(\mathcal{A})=r_{i}(\mathcal{B}),\quad  s_{ij}(\mathcal{A})=s_{ij}(\mathcal{B}), \quad \forall i,j\in N.
\end{equation*}
We assume $\lambda\notin \mathcal{F}(\mathcal{A})$, which indicate that there is $\lambda\notin \mathcal{F}_{i}(\mathcal{A}), \forall i\in N $. In other words, we obtain
\begin{equation*}
  |\lambda-a_{i\cdots i}|-s_{ii}(\mathcal{A})> \sum\limits_{j=1,j\neq i}^{n}|s_{ij}(\mathcal{A})|, \quad \forall i\in N.
\end{equation*}
Equivalently, we have
\begin{equation*}
  |b_{i\cdots i}|- s_{ii}(\mathcal{B})> \sum\limits_{j=1,j\neq i}^{n}|s_{ij}(\mathcal{B})|, \quad \forall i\in N.
\end{equation*}
According to equalities \eqref{beq2} and \eqref{defd},
$\mathcal{B}$ is a strictly diagonally dominant tensor. Moreover, there is $0\notin \sigma(\mathcal{B})$,  which is in contradiction with the $0\in \sigma(\mathcal{B})$. The proof is completed.
\end{proof}

\begin{remark}
It's worth noting that the set ${\mathcal{F}(\mathcal{A})}$ is consistent with classical Ger$\breve{s}$gorin set $\Gamma(\mathcal{A})$ in \cite{Qi2005}.
\end{remark}

Furthermore, Li et al. \cite{LI2014N} demonstrated that the well-established matrix Brauer set, often contained within the Ger$\breve{s}$gorin sets, lacks generalizability to higher-order tensors.

\begin{theorem}[Brauer's Ovals of Cassini \cite{Varga}]
Let $A=(a_{ij})\in \mathbb{C}^{n,n}$ be a $n\times n$ complex matrix, $n\geq 2$, and $\sigma(A)$ be the spectrum of $A$. Then
\begin{equation*}
  \sigma(A)\subseteq \Theta(A)=\bigcup\limits_{\substack{i,j\in N \\ j\neq i}}\Theta_{ij}(A),
\end{equation*}
where $\Theta_{ij}(A)=\{z\in \mathbb{C}:|z-a_{ii}||z-a_{jj}|\leq r_{i}(A)r_{j}(A)\}$, $r_{i}(A)=\sum\limits_{k\neq i}|a_{ik}|$.
\end{theorem}

However, with the aid of the connections between tensors and tensor-generated matrices, along with the principles of double diagonal dominance and matrix eigenvalue distribution, we are able to derive the following modified Brauer set for higher-order tensor $\mathcal{A}$.

\begin{theorem}\label{DubleO}
Let $\mathcal{A}=(a_{i_1i_2\cdots i_m}) \in \mathbb{C}^{[m\times n]}$. Then
\begin{equation*}
  \sigma(\mathcal{A})\subseteq \mathcal{D}(\mathcal{A})=\bigcup_{\substack{i,j\in N \\ j\neq i}} \mathcal{D}_{i,j}(\mathcal{A}),
\end{equation*}
where
\begin{equation*}
  \mathcal{D}_{i,j}(\mathcal{A})=\{z\in \mathbb{C}:[|z-a_{ii\cdots i}|-s_{ii}(\mathcal{A})][|z-a_{jj\cdots j}|-s_{jj}(\mathcal{A})]\leq P_{i}(A)P_{j}(A)\}.
\end{equation*}
\end{theorem}
\begin{proof}
For any $\lambda \in \sigma(\mathcal{A})$, let $\mathcal{B}=(b_{i_1 i_2\cdots i_m})=\lambda\mathcal{I}-\mathcal{A}$. Then
\begin{equation*}
  0\in \sigma(\mathcal{B}),\quad  r_{i}(\mathcal{A})=r_{i}(\mathcal{B}),\quad  s_{ij}(\mathcal{A})=s_{ij}(\mathcal{B}), \quad \forall i,j\in N.
\end{equation*}
We assume $\lambda\notin \mathcal{D}(\mathcal{A})$, which indicate that there is $\lambda\notin \mathcal{D}_{i,j}(\mathcal{A}), \forall i,j\in N $. In other words, we get
\begin{equation*}
  [|\lambda-a_{ii\cdots i}|-s_{ii}(\mathcal{A})][|\lambda-a_{jj\cdots j}|-s_{jj}(\mathcal{A})]> P_{i}(A)P_{j}(A), \quad \forall i,j\in N,
\end{equation*}
that is,
\begin{equation*}
  [|b_{ii\cdots i}|-s_{ii}(\mathcal{B})][|b_{jj\cdots j}|-s_{jj}(\mathcal{B})]> P_{i}(B)P_{j}(B), \quad \forall i,j\in N.
\end{equation*}
According to equalities \eqref{beq2} and \eqref{defdd},
the tensor-generated matrix $B$ of tensor $\mathcal{B}$ is a strictly doubly diagonally dominant matrix, that is, $\mathcal{B}$ is a $H$-tensor. Moreover, there is $0\notin \sigma(\mathcal{B})$,  which is in contradiction with the $0\in \sigma(\mathcal{B})$. The proof is completed.
\end{proof}

As the radii of the Ger$\breve{s}$gorin discs rely on both the deleted row sums $r_{i}(A)$ and the deleted column sums $c_{i}(A)$, Ostrowski \cite{Varga} introduces the well-known matrix eigenvalue inclusion set, known as the Ostrowski set.

\begin{theorem}[Ostrowski \cite{Varga}]
Let $A=(a_{ij})\in \mathbb{C}^{n,n}$ be a $n\times n$ complex matrix, $n\geq 2$, $0\leq \gamma \leq 1$, and $\sigma(A)$ be the spectrum of $A$. Then
\begin{equation*}
  \sigma(A)\subseteq \Omega(A)=\bigcup\limits_{i=1}^{n}\Omega_{ij}(A),
\end{equation*}
where $\Theta_{ij}(A)=\{z\in \mathbb{C}:|z-a_{ii}|\leq (r_{i}(A))^{\gamma}(c_{i}(A))^{1-\gamma}\}$, $r_{i}(A)=\sum\limits_{j\neq i}|a_{ij}|$, $c_{i}(A)=\sum\limits_{j\neq i}|a_{ji}|$.
\end{theorem}


While direct generalization of Ostrowski sets to higher-order tensors, akin to matrix eigenvalue Brauer sets, is unfeasible, the modified Ostrowski sets tailored for higher-order tensors $\mathcal{A} \in \mathbb{C}^{[m\times n]}$ can be obtained by leveraging the correlation between the tensor $\mathcal{A} \in \mathbb{C}^{[m\times n]}$ and its corresponding tensor-generated matrix $A\in \mathbb{C}^{n,n}$. This process involves establishing a connection between $\gamma$-diagonally dominant matrices and the distribution of matrix eigenvalues.

\begin{theorem}\label{POst}
Let $\mathcal{A}=(a_{i_1i_2\cdots i_m}) \in \mathbb{C}^{[m\times n]}$, $n\geq 2$, $0\leq \gamma \leq 1$. Then
\begin{equation*}
  \sigma(\mathcal{A})\subseteq \mathcal{O}(\mathcal{A})=\bigcup_{i\in N} \mathcal{O}_{i}(\mathcal{A}),
\end{equation*}
where
\begin{equation*}
  \mathcal{O}_{i}(\mathcal{A})=\{z\in \mathbb{C}:|z-|a_{i\cdots i}|+ s_{ii}(\mathcal{A})|\leq [P_{i}(A)]^{\gamma }[Q_{i}(A)]^{1-\gamma}\}.
\end{equation*}
\end{theorem}
\begin{proof}
For any $\lambda \in \sigma(\mathcal{A})$, let $\mathcal{B}=(b_{i_1 i_2\cdots i_m})=\lambda\mathcal{I}-\mathcal{A}$. Then
\begin{equation*}
  0\in \sigma(\mathcal{B}),\quad  r_{i}(\mathcal{A})=r_{i}(\mathcal{B}),\quad  s_{ij}(\mathcal{A})=s_{ij}(\mathcal{B}), \quad \forall i,j\in N.
\end{equation*}
We assume $\lambda\notin \mathcal{O}(\mathcal{A})$, which indicate that there is $\lambda\notin \mathcal{O}_{i}(\mathcal{A}), \forall i\in N $. In other words, we obtain
\begin{equation*}
  [|\lambda-a_{i\cdots i}|- s_{ii}(\mathcal{A})]> [P_{i}(A)]^{\gamma }[Q_{i}(A)]^{1-\gamma}, \quad \forall i\in N.
\end{equation*}
That is to say
\begin{equation*}
  [|b_{i\cdots i}|- s_{ii}(\mathcal{B})]> [P_{i}(B)]^{\gamma }[Q_{i}(B)]^{1-\gamma}, \quad \forall i\in N.
\end{equation*}
According to equalities \eqref{beq2} and \eqref{defo1},
the tensor-generated matrix $B$ of tensor $\mathcal{B}$ is a strictly  $\gamma$-diagonally dominant matrix, that is to say, $\mathcal{B}$ is a $H$-tensor. Moreover, there is $0\notin \sigma(\mathcal{B})$,  which is in contradiction with the $0\in \sigma(\mathcal{B})$. The proof is completed.
\end{proof}

\begin{lemma}\label{lem4.8}
If $p,q>0$ and $0\leq \gamma \leq 1$, then
\begin{equation*}
  p^{\gamma}q^{1-\gamma}\leq \gamma p+(1-\gamma)q.
\end{equation*}
The equality holds if and only if $p=q$.
\end{lemma}

Subsequent to Theorem \ref{POst} and Lemma \ref{lem4.8}, the ensuing Corollary is deduced.

\begin{corollary}\label{OP}
Let $\mathcal{A}=(a_{i_1i_2\cdots i_m}) \in \mathbb{C}^{[m\times n]}$, $n\geq 2$, $0\leq \gamma \leq 1$. Then
\begin{equation*}
  \sigma(\mathcal{A})\subseteq \mathcal{W}(\mathcal{A})=\bigcup_{i\in N} \mathcal{W}_{i}(\mathcal{A}),
\end{equation*}
where
\begin{equation*}
  \mathcal{W}_{i}(\mathcal{A})=\{z\in \mathbb{C}:[|z-a_{i\cdots i}|-s_{ii}(\mathcal{A})]\leq \gamma P_{i}(A)+(1-\gamma)Q_{i}(A)\}.
\end{equation*}
\end{corollary}

Let set $S$ be any nonempty subset of $N=\{1,2,\cdots,n\}$, $n\geq 2$, where $\overline{S}=N\setminus S$ is the complement of the subset $S$ in $N$. For any given matrix $A=(a_{ij})\in \mathbb{C}^{n,n}$, each row sum $r_{i}(A)$ is divided into $S$ and $\overline{S}$ parts:
\begin{equation*}
\begin{cases}
 r_{i}(A)=\sum\limits_{j\in N\setminus \{i\}}|a_{ij}|=r_{i}^{S}(A)+r_{i}^{\overline{S}}(A),\\
 r_{i}^{S}(A)=\sum\limits_{j\in S\setminus \{i\}}|a_{ij}|,\quad r_{i}^{\overline{S}}(A)=r_{i}(A)-r_{i}^{S}(A)=\sum\limits_{j\in \overline{S}\setminus \{i\}}|a_{ij}|, i\in N.
\end{cases}
\end{equation*}

\begin{definition}\label{de4.4.1}
Let set $S$ be any nonempty subset of $N=\{1,2,\cdots,n\}$, $n\geq 2$, for any given matrix $A=(a_{ij})\in \mathbb{C}^{n,n}$, then the matrix $A$ is strictly $S$-diagonally dominant, if satisfied:
\begin{equation*}
\begin{cases}
|a_{ii}|>r_{i}^{S}(A), \forall i\in S,\\
|a_{ii}|-r_{i}^{S}(A)\big)\cdot\big(|a_{jj}|-r_{j}^{\overline{S}}(A)\big)> r_{i}^{\overline{S}}(A)r_{j}^{S}(A), i\in S, j\in \overline{S}.
\end{cases}
\end{equation*}
\end{definition}


The matrix eigenvalue $S$-type inclusion theorem is a prominent classical theory within the field of matrix eigenvalue inclusion theorems. This theorem directly follows from the matrix eigenvalue inclusion theorem based on $S$-diagonally dominant matrices.

\begin{theorem}\rm{\cite{Varga}}
Let $S$ be any nonempty subset of $N=\{1,2,\cdots,n\}$, $n\geq 2$, where $\overline{S}=N\setminus S$ is the complement of the subset $S$ in $N$. Then, for any given matrix $A=(a_{ij})\in \mathbb{C}^{n,n}$, define the Ger$\breve{s}$gorin-type disks
\begin{equation*}
  \Gamma_{i}^{S}(A)=\{z\in \mathbb{C}:|z-a_{ii}|\leq r_{i}^{S}(A)\}, \forall i\in S,
\end{equation*}
and the sets
\begin{equation*}
  V_{i,j}^{S}(A)=\{z\in \mathbb{C}:\big(|z-a_{ii}|-r_{i}^{S}(A)\big)\cdot\big(|z-a_{jj}|-r_{j}^{\overline{S}}(A)\big)\leq r_{i}^{\overline{S}}(A)r_{j}^{S}(A)\},
\end{equation*}
($\forall i\in S,\forall j\in \overline{S}$). Then,
\begin{equation*}
  \sigma(A)\subseteq C^{S}(A)=\bigg(\bigcup_{i\in S}\Gamma_{i}^{S}(A)\bigg)\bigcup
  \bigg(\bigcup_{i\in S,j\in \overline{S}}V_{i,j}^{S}(A)\bigg).
\end{equation*}
\end{theorem}

\begin{equation*}
\begin{cases}
 P_{i}(\mathcal{A})=\sum\limits_{j\in N\setminus \{i\}}|a_{ij}|=r_{i}^{S}(\mathcal{A})+r_{i}^{\overline{S}}(\mathcal{A}),\\
 r_{i}^{S}(\mathcal{A})=\sum\limits_{j\in S\setminus \{i\}}|s_{ij}|,\quad r_{i}^{\overline{S}}(\mathcal{A})=P_{i}(\mathcal{A})-r_{i}^{S}(\mathcal{A}), i\in N.
\end{cases}
\end{equation*}


Exploiting the established relationship between tensors and tensor-generated matrices in the previous section, we expand the classical matrix eigenvalue $S$-type inclusion theorem to encompass higher-order tensor $H$-eigenvalues. Consequently, we present a modified $S$-type inclusion theorem specifically tailored for tensor $H$-eigenvalues. This theorem enables us to effectively characterize the inclusion of $H$-eigenvalues within higher-order tensors by leveraging the connection between tensors and their tensor-generated matrices.
\begin{theorem}\label{thm4.4.7}
Let $\mathcal{A}=(a_{i_1i_2\cdots i_m}) \in \mathbb{C}^{[m\times n]}$ be a $m$ order $n$ dimension tensor and $S$ be any nonempty subset of $N=\{1,2,\cdots,n\}$, $n\geq 2$, where $\overline{S}=N\setminus S$ is the complement of the subset $S$ in $N$. Then
\begin{equation*}
  \sigma(\mathcal{A})\subseteq \mathcal{C}^{S}(\mathcal{A})=\bigg(\bigcup_{i\in S}\Gamma_{i}^{S}(\mathcal{A})\bigg)\bigcup
  \bigg(\bigcup_{i\in S,j\in \overline{S}}\mathcal{C}_{i,j}^{S}(\mathcal{A})\bigg)
\end{equation*}
where
\begin{equation*}
  \Gamma_{i}^{S}(\mathcal{A})=\{z\in \mathbb{C}:||z-a_{i\cdots i}|-s_{ii}(\mathcal{A})|\leq r_{i}^{S}(\mathcal{A})\}, \forall i\in S,
\end{equation*}
\begin{equation*}
\begin{split}
  \mathcal{C}_{i,j}^{S}(\mathcal{A})=\{z\in \mathbb{C}:&[||z-a_{i\cdots i}|-s_{ii}(\mathcal{A})|-r_{i}^{S}(\mathcal{A})]
  [||z-a_{j\cdots j}|-s_{jj}(\mathcal{A})|-r_{j}^{\overline{S}}(\mathcal{A})]\\
  &\leq  r_{i}^{\overline{S}}(\mathcal{A})r_{j}^{S}(\mathcal{A})\}.
\end{split}
\end{equation*}
\end{theorem}
\begin{proof}
Let $\mathcal{B}=(b_{i_1 i_2\cdots i_m})=\lambda\mathcal{I}-\mathcal{A}$, for any $\lambda \in \sigma(\mathcal{A})$, then
\begin{equation*}
  0\in \sigma(\mathcal{B}),\quad  r_{i}(\mathcal{A})=r_{i}(\mathcal{B}),\quad  s_{ij}(\mathcal{A})=s_{ij}(\mathcal{B}), \quad \forall i,j\in N,
\end{equation*}
and
\begin{equation*}
  r_{i}^{S}(\mathcal{A})=r_{i}^{S}(\mathcal{B}), r_{i}^{\overline{S}}(\mathcal{A})=r_{i}^{\overline{S}}(\mathcal{B}),\quad \forall i\in N.
\end{equation*}
We assume that $\lambda\notin \mathcal{C}^{S}(\mathcal{A})$, which shows that $\lambda\notin \bigg(\bigcup_{i\in S}\Gamma_{i}^{S}(\mathcal{A})\bigg)\bigcup
  \bigg(\bigcup_{i\in S,j\in \overline{S}}\mathcal{C}_{i,j}^{S}(\mathcal{A})\bigg), \forall i\in N $. In other words, we have
\begin{equation*}
\begin{cases}
||\lambda-a_{i\cdots i}|-s_{ii}(\mathcal{A})|> r_{i}^{S}(\mathcal{A}),\\
[||\lambda-a_{i\cdots i}|-s_{ii}(\mathcal{A})|-r_{i}^{S}(\mathcal{A})]
[||\lambda-a_{j\cdots j}|-s_{jj}(\mathcal{A})|-r_{j}^{\overline{S}}(\mathcal{A})]>r_{i}^{\overline{S}}(\mathcal{A})r_{j}^{S}(\mathcal{A}),
\end{cases}
\end{equation*}
equivalently,
\begin{equation*}
\begin{cases}
||b_{i\cdots i}|-s_{ii}(\mathcal{B})|> r_{i}^{S}(\mathcal{B}),\\
[||b_{i\cdots i}|-s_{ii}(\mathcal{B})|-r_{i}^{S}(\mathcal{B})]
[||b_{j\cdots j}|-s_{jj}(\mathcal{B})|-r_{j}^{\overline{S}}(\mathcal{B})]>r_{i}^{\overline{S}}(\mathcal{B})r_{j}^{S}(\mathcal{B}),
\end{cases}
\end{equation*}
From Definition \ref{TGM} and Definition \ref{de4.4.1}, the tensor-generated matrix $B$ of tensor $\mathcal{B}\in \mathbb{C}^{[m\times n]}$ is $S$-strictly diagonally dominant matrix. Hence, $\mathcal{B}$ is a $H$-tensor based on Theorem \ref{HH}. Furthermore, it is result in $0\notin \sigma(\mathcal{B})$ , which is contradicted with $0\in \sigma(\mathcal{B})$. The proof is completed.
\end{proof}

When $S=\{i\}$ for $i\in N$, then Theorem \ref{thm4.4.7} has the following case.
\begin{corollary}\label{cor4.4.2}
Let $\mathcal{A}=(a_{i_1i_2\cdots i_m}) \in \mathbb{C}^{[m\times n]}$. Then
\begin{equation*}
  \sigma(\mathcal{A})\subseteq \mathcal{C}^{S^{i}}(\mathcal{A})=\bigcup_{i\in S,j\in \overline{S}}\mathcal{C}_{i,j}^{S^{i}}(\mathcal{A})
\end{equation*}
where
\begin{equation*}
  \mathcal{C}_{i,j}^{S^{i}}(\mathcal{A})=\{z\in \mathbb{C}:||z-a_{i\cdots i}|-s_{ii}(\mathcal{A})|
  \big(||z-a_{j\cdots j}|-s_{jj}(\mathcal{A})|-P_{j}(\mathcal{A})+s_{ji}(\mathcal{A})\big)
  \leq  P_{i}(\mathcal{A})s_{ji}(\mathcal{A})\}.
\end{equation*}
\end{corollary}

%
%
%

It is worth highlighting that our method is not only effective for the aforementioned matrix eigenvalue localization sets but also applicable to a wide range of matrix inclusion sets, including other Brauers-type inclusion sets and their modified versions \cite{Varga}, Cvetkovi$\acute{c}$-I(II) sets \cite{Ljiljana1,Ljiljana}, and related references.

Theorem \ref{thmdir} reveals the consistent relationship between the irreducibility of the tensor-generated matrix and the weak irreducibility of the original tensor. Consequently, when dealing with irreducible tensor-generated matrices, our method is also suitable for Brualdi-type inclusion sets \cite{Brualdi,Varga}. Although we will omit the derivation process here, similar principles are applied.

To summarize,  we transform the problem of tensor eigenvalue localization into matrix eigenvalue localization. This allows us to leverage the mature theory of matrix eigenvalue localization and apply it to the investigation of tensor eigenvalue localization.

%

\subsection{Numerical examples}
Finally, we provide numerical examples to demonstrate the validity and effectiveness of the tensor $H$-eigenvalue inclusion theorem derived in the above conclusions.

We analyze Example 1.1 from \cite{LI2014N}, demonstrating that Brauer's Ovals of Cassini sets, typically linked with matrix eigenvalues, do not directly apply to an inclusion theorem for higher-order tensor $H$-eigenvalues.

\begin{example}
Let $\mathcal{A} \in \mathbb{R}^{[4\times 2]}$, where the non-zero elements are as follows:
\begin{equation*}
  a_{1111}=7, a_{1112}=a_{1121}=a_{1211}=a_{2111}=-2;
\end{equation*}
\begin{equation*}
  a_{2222}=6, a_{2221}=a_{2212}=a_{2122}=a_{1222}=-1;
\end{equation*}
otherwise $a_{i_1i_2i_3 i_4}=0$.
\end{example}

For a 4 order 2 dimension tensor, we can relate the elements of this tensor to a $2\times 2$ tensor-generated matrix \eqref{matrix A} due to the connection between the tensor and the tensor-generated matrix. Hence, we obtain:
\begin{equation*}
  A=\left(
      \begin{array}{cc}
        |7|-|4| & 3 \\
        3 & |6|-|2| \\
      \end{array}
    \right).
\end{equation*}

According to Theorem \ref{DubleO}, there are
\begin{equation*}
 \sigma(\mathcal{A})\subseteq \mathcal{\bar{O}}(\mathcal{A})=\{z\in \mathbb{C}: 0.4586\leq z \leq 12.8541\}.
\end{equation*}

In fact, the eigenvalues of tensor $\mathcal{A}=(a_{i_1i_2i_3 i_4})$ are
$\sigma(\mathcal{A})=\{12.7389, 0.4725\}.$
This demonstrates the accuracy of the adjusted ovality theorem for higher-order tensor eigenvalues.

\begin{example}
Consider $\mathcal{A}\in \mathbb{R}^{[4\times 4]}$ as a real tensor of order 4 and dimension 4 with specific element assignments:
\begin{equation*}
  a_{1111}=10, \quad a_{2222}=8, \quad a_{3333}=7, \quad a_{4444}=5,\quad a_{1333}=a_{1444}=1,
\end{equation*}
\begin{equation*}
    a_{1211}=a_{1113}=a_{1141}=1, \quad a_{1332}=a_{1442}=a_{1232}=a_{1234}=a_{1321}=a_{1214}=1,\quad a_{2333}=a_{2444}=1,
\end{equation*}
\begin{equation*}
   \quad a_{2112}=a_{2234}=a_{2113}=a_{2343}=a_{2123}=1, \quad a_{3222}=1, \quad a_{3111}=1, \quad a_{3121}=a_{3434}=a_{3123}=1,
\end{equation*}
\begin{equation*}
  a_{4222}=1, \quad a_{4111}=1, \quad a_{4121}=a_{4334}=1, \quad \text{and other entries are zero.}
\end{equation*}
\end{example}

The elements of $\mathcal{A}$ in $\mathbb{R}^{[4\times 4]}$ can be associated with a $4\times 4$ tensor-generated matrix \eqref{matrix A}, resulting in:
\begin{equation*}
  A=\left(
             \begin{array}{cccc}
              10-\frac{8}{3} & \frac{8}{3} & 3 & \frac{8}{3} \\
               \frac{5}{3} & 8-1 & \frac{8}{3} & \frac{5}{3} \\
               2 & \frac{5}{3} & 7-\frac{2}{3} & \frac{2}{3} \\
               \frac{5}{3} & \frac{4}{3} & \frac{2}{3} & 5-\frac{1}{3} \\
             \end{array}
           \right).
\end{equation*}

The numerical results from Theorem \ref{MG}, Theorem \ref{DubleO}, Corollary \ref{OP}, Theorem \ref{thm4.4.7}, and Corollary \ref{cor4.4.2}, as well as findings from \cite{Qi2005,LI2014N,LCQ2016S,RJZ2016}, are presented in Table 1. This illustrates that our eigenvalue localization set, encompassing Theorem \ref{MG}, Theorem \ref{DubleO}, Corollary \ref{OP}, Theorem \ref{thm4.4.7}, and Corollary \ref{cor4.4.2}, effectively captures all eigenvalues. Additionally, results from our methods in the mentioned examples reveal the enhanced precision of modified Brauer-type, Ostrowski, and $S$-type inclusion sets for tensor eigenvalues from the tensor-generated matrix compared to existing sets. Specifically, all the $H$-eigenvalues of the tensor $\mathcal{A}$ can be computed as follows:
$\sigma(\mathcal{A})=\{4.4858, 7.3107, 9.7718, 15.2641\}.$

\begin{table}[h]\label{table}
\caption{Lower and upper bounds of $\sigma(\mathcal{A})$.}
\begin{center}
\begin{tabular}{|c|c|c|}
  \hline
  \qquad Method \qquad& \qquad Lower bounds \qquad & \qquad Upper bounds \qquad\\
  \hline
  Theorem 6 in \cite{Qi2005} & -1 & 21 \\
  Theorem \ref{MG} & -1 & 21 \\
  Theorem 2.2 in \cite{LI2014N} & -0.7016 & 20.3739 \\
  Theorem 4 in \cite{LCQ2016S} & -0.4641 & 19.7823 \\
  Theorem 11 in \cite{RJZ2016} & -0.3723 & 19.6847 \\
  Theorem \ref{DubleO} & 0.0936 & 18.1382 \\
  Theorem \ref{POst}, $\gamma=0.5$ & 0.3849 & 19.3333 \\
  Theorem \ref{POst}, $\gamma=0.04$ &-0.2717 & 18.0961 \\
  Corollary \ref{OP},$\gamma=0.5$ & 0.3333  & 19.5000 \\
  Corollary \ref{OP},$\gamma=0.04$ & -0.2800 & 18.1200 \\
  Theorem \ref{thm4.4.7},$S=\{1,2\}$ & 0.5811  & 17.5803 \\
  Corollary \ref{cor4.4.2} & -0.4741 & 19.8130\\
  \hline
\end{tabular}
\end{center}
\end{table}

\section*{Conclusions}

Multipartite quantum scenarios serve as crucial and complex resources within the domain of quantum information science. Tensors establish a robust framework for the representation of multipartite quantum systems. This study introduces the novel concept of tensor-generated matrices, which define the relationships between an $m$-th order $n$-dimensional tensor and an $n$-dimensional square matrix. By establishing these connections, we illustrate that designating the tensor-generated matrix as an $H$-matrix implies the classification of the original tensor as an $H$-tensor. Additionally, we explore various analogous properties shared by both the original tensor and the tensor-generated matrix, encompassing attributes like weak irreducibility, weakly chained diagonal dominance, and (strong) symmetry. These discoveries offer a methodology to convert intricate tensor issues into matrix formulations within specific contexts, a particularly pertinent approach given the inherent NP-hard complexity of most tensor problems.

Subsequently, we delve into the utilization of tensor-generated matrices to analyze the classicality of spin states. Capitalizing on the tensor representation, we introduce classicality criteria tailored for (strongly) symmetric spin-$j$ states, potentially introducing novel perspectives into the examination of multipartite quantum resources. Furthermore, we extend classical matrix eigenvalue inclusion sets to higher-order tensor $H$-eigenvalues, a challenge often encountered with higher-order tensors. Consequently, we propose representative tensor $H$-eigenvalue inclusion sets, such as modified versions of Brauer's Ovals of Cassini sets, Ostrowski sets, and $S$-type inclusion sets.

\section*{Data availibility}
All data generated or analysed during this study are included in this published article.

\section*{Declarations}
Conflict of interest The authors declare that they have no known competing financial interests or personal
relationships that could have appeared to influence the work reported in this paper.

\bibliographystyle{elsarticle-num}
  \bibliography{H-tensor}

\begin{thebibliography}{10}
\expandafter\ifx\csname url\endcsname\relax
  \def\url#1{\texttt{#1}}\fi
\expandafter\ifx\csname urlprefix\endcsname\relax\def\urlprefix{URL }\fi
\expandafter\ifx\csname href\endcsname\relax
  \def\href#1#2{#2} \def\path#1{#1}\fi

\bibitem{CYN2017L1}
Y.~Chen, L.~Qi, E.~G. Virga, Octupolar tensors for liquid crystals, J. Phys. A:
  Math. Theor. 51~(2) (2018) 025206.

\bibitem{liquid2}
G.~Gaeta, E.~G. Virga, Octupolar order in three dimensions, Eur. Phys. J. E.
  39~(11) (2016) 113.

\bibitem{liquid3}
E.~G. Virga, Octupolar order in two dimensions, Eur. Phys. J. E. 38~(6) (2015)
  63.

\bibitem{East}
L.~Xiong, J.~Liu, A new {C}-eigenvalue localisation set for piezoelectric-type
  tensors, East Asian J. Appl. Math. 10~(1) (2020) 123--134.

\bibitem{Markov2}
W.~Li, M.~K. Ng, On the limiting probability distribution of a transition
  probability tensor, Linear. Multilinear. Algebra. 62~(3) (2014) 362--385.

\bibitem{data1}
L.~Lathauwer, B.~Moor, J.~Vandewalle, On the best rank-1 and rank-({R}1,{R}2, .
  . . ,{RN}) approximation of higher-order tensors, SIAM.J.Matrix.Anal.A. 21
  (2000) 1324--1342.

\bibitem{data2}
L.~Qi, W.~Sun, Y.~Wang, Numerical multilinear.algebra.appl., Front. Math. China
  2 (2007) 501--526.

\bibitem{KE2016}
R.~Ke, W.~Li, M.~K. Ng, Numerical ranges of tensors, Linear.Algebra.Appl. 508
  (2016) 100--132.

\bibitem{GME2003}
T.-C. Wei, P.~M. Goldbart, Geometric measure of entanglement and applications
  to bipartite and multipartite quantum states, Phys.Rev.A. 68 (2003) 042307.

\bibitem{hu2016}
S.~Hu, L.~Qi, G.~Zhang, Computing the geometric measure of entanglement of
  multipartite pure states by means of non-negative tensors, Phys.Rev.A. 93~(1)
  (2016) 012304.

\bibitem{quantuminformation}
M.~A. Nielsen, I.~L. Chuang, Quantum Computing and Quantum Information,
  Cambridge University Press: Cambridge, 2000.

\bibitem{COAM}
L.~Xiong, J.~Liu, {Z}-eigenvalue inclusion theorem of tensors and the geometric
  measure of entanglement ofmultipartite pure states, Comput. Appl. Math.
  39~(135) (2020).

\bibitem{ACAP}
L.~Xiong, J.~Liu, Further results for {Z}-eigenvalue localization theorem for
  higher-order tensors and their applications, Acta Appl. Math. 170~(11) (2020)
  229--264.

\bibitem{Niq2008}
Q.~Ni, L.~Qi, F.~Wang, An eigenvalue method for testing positive definiteness
  of a multivariate form, IEEE Trans. Autom. Control. 53~(5) (2008) 1096--1107.

\bibitem{JMI}
J.~Liu, L.~Xiong, Exponential type locally generalized strictly double
  diagonally tensors and eigenvalue localization, J. Math. Inequal. 14~(3)
  (2020) 611--629.

\bibitem{PhysRevLett.88.127902}
N.~J. Cerf, M.~Bourennane, A.~Karlsson, N.~Gisin, Security of quantum key
  distribution using $\mathit{d}$-level systems, Phys.Rev.Lett. 88 (2002)
  127902.

\bibitem{PhysRevA.82.030301}
L.~Sheridan, V.~Scarani, Security proof for quantum key distribution using
  qudit systems, Phys.Rev.A. 82 (2010) 030301.

\bibitem{PhysRevLett.83.5162}
D.~S. Abrams, S.~Lloyd, Quantum algorithm providing exponential speed increase
  for finding eigenvalues and eigenvectors, Phys.Rev.Lett. 83 (1999)
  5162--5165.

\bibitem{RevModPhys.80.1061}
A.~Das, B.~K. Chakrabarti, Colloquium: Quantum annealing and analog quantum
  computation, Rev.Mod. Phys. 80 (2008) 1061--1081.

\bibitem{chen2010tensor}
L.~Chen, E.~Chitambar, R.~Duan, Z.~Ji, A.~Winter, Tensor rank and stochastic
  entanglement catalysis for multipartite pure states, Phys.Rev.Lett. 105~(20)
  (2010) 200501.

\bibitem{chitambar2008tripartite}
E.~Chitambar, R.~Duan, Y.~Shi, Tripartite entanglement transformations and
  tensor rank, Phys.Rev.Lett. 101~(14) (2008) 140502.

\bibitem{bruzda2024rank}
W.~Bruzda, S.~Friedland, K.~{\.Z}yczkowski, Rank of a tensor and quantum
  entanglement, Linear. Multilinear. Algebra. 72~(11) (2024) 1796--1859.

\bibitem{xiong2022cauchy}
L.~Xiong, Y.~Wu, J.~Liu, Z.~Jiang, Q.~Qin, Cauchy tensor and the classicality
  and separability condition of spin states, Results Phys. 40 (2022) 105824.

\bibitem{xiong2022geometric}
L.~Xiong, J.~Liu, Q.~Qin, The geometric measure of entanglement of multipartite
  states and the {Z}-eigenvalue of tensors, Quantum Inf. Process. 21~(3) (2022)
  102.

\bibitem{PRXQuantum.2.030305}
J.~Lee, D.~W. Berry, C.~Gidney, W.~J. Huggins, J.~R. McClean, N.~Wiebe,
  R.~Babbush, Even more efficient quantum computations of chemistry through
  tensor hypercontraction, PRX Quantum 2 (2021) 030305.

\bibitem{2014T}
O.~Giraud, D.~Braun, D.~Baguette, T.~Bastin, J.~Martin, Tensor representation
  of spin states, Phys.Rev.Lett. 114 (2015) 080401.

\bibitem{PhysRevA.94.042324}
F.~Bohnet-Waldraff, D.~Braun, O.~Giraud, Tensor eigenvalues and entanglement of
  symmetric states, Phys.Rev.A. 94 (2016) 042324.

\bibitem{RevModPhys.89.041003}
A.~Streltsov, G.~Adesso, M.~B. Plenio, Colloquium: Quantum coherence as a
  resource, Rev.Mod. Phys. 89 (2017) 041003.

\bibitem{optics}
E.~C.~G. Sudarshan, Equivalence of semiclassical and quantum mechanical
  descriptions of statistical light beams, Phys.Rev.Lett. 10 (1963) 277--279.

\bibitem{optics2}
R.~J. Glauber, Coherent and incoherent states of the radiation field,
  Phys.Rev.Lett. 131 (1963) 2766--2788.

\bibitem{gao2015coherent}
W.~Gao, A.~Imamoglu, H.~Bernien, R.~Hanson, Coherent manipulation, measurement
  and entanglement of individual solid-state spins using optical fields,
  Nat.Photonics 9~(6) (2015) 363--373.

\bibitem{PhysRevLett.111.250404}
F.~G. S.~L. Brand\~ao, M.~Horodecki, J.~Oppenheim, J.~M. Renes, R.~W. Spekkens,
  Resource theory of quantum states out of thermal equilibrium, Phys. Rev.
  Lett. 111 (2013) 250404.

\bibitem{PhysRevLett.113.140401}
T.~Baumgratz, M.~Cramer, M.~B. Plenio, Quantifying coherence, Phys.Rev.Lett.
  113 (2014) 140401.

\bibitem{PhysRevLett.116.150502}
C.~Napoli, T.~R. Bromley, M.~Cianciaruso, M.~Piani, N.~Johnston, G.~Adesso,
  Robustness of coherence: An operational and observable measure of quantum
  coherence, Phys.Rev.Lett. 116 (2016) 150502.

\bibitem{PhysRevA.98.022328}
N.~Johnston, C.-K. Li, S.~Plosker, Y.-T. Poon, B.~Regula, Evaluating the
  robustness of $k$-coherence and $k$-entanglement, Phys.Rev.A. 98 (2018)
  022328.

\bibitem{coherent}
A.~M. Perelomov, Coherent states for arbitrary lie group., Commun. Math. Phys.
  26 (1972) 222--236.

\bibitem{1986Generalized}
A.~Perelomov, Generalized Coherent States and Their Applications,
  Springer-Verlag, Berlin, 1986.

\bibitem{Classicality2008}
O.~Giraud, P.~Braun, D.~Braun, Classicality of spin states, Phys.Rev.A. 78
  (2008) 042112.

\bibitem{PhysRevA.96.032312}
F.~Bohnet-Waldraff, D.~Braun, O.~Giraud, Entanglement and the truncated moment
  problem, Phys.Rev.A. 96 (2017) 032312.

\bibitem{quantumness2010}
G.~Olivier, B.~Petr, B.~Daniel, Quantifying quantumness and the quest for
  queens of quantum, New J. Phys. 12~(6) (2010) 063005.

\bibitem{quantumness2016}
F.~Bohnet-Waldraff, D.~Braun, O.~Giraud, Quantumness of spin-1 states,
  Phys.Rev.A. 93 (2016) 012104.

\bibitem{2010H}
J.~Martin, O.~Giraud, P.~A. Braun, D.~Braun, T.~Bastin, Multiqubit symmetric
  states with high geometric entanglement, Phys.Rev.A. 81 (2010) 062347.

\bibitem{Eisert_2003}
J.~Eisert, K.~Audenaert, M.~B. Plenio, Remarks on entanglement measures and
  non-local state distinguishability, J. Phys. A: Math. Gen. 36~(20) (2003)
  5605--5615.

\bibitem{hillar2013most}
C.~J. Hillar, L.-H. Lim, Most tensor problems are {NP}-hard, Journal of the ACM
  60~(6) (2013) 1--39.

\bibitem{Power2011}
T.~G. Kolda, J.~R. Mayo., Shifted power method for computing tensor eigenpairs,
  SIAM.J.Matrix.Anal.A. 32~(4) (2011) 1095--1124.

\bibitem{Power2014}
T.~G. Kolda, J.~R. Mayo., An adaptive shifted power method for computing
  generalized tensor eigenpairs, SIAM.J.Matrix.Anal.A. 35~(4) (2014)
  1563--1581.

\bibitem{Cui2014}
C.-F. Cui, Y.-H. Dai, J.~Nie, All real eigenvalues of symmetric tensors,
  SIAM.J.Matrix.Anal.A. 35~(4) (2014) 1582--1601.

\bibitem{Homotopy1}
L.~Chen, L.~Han, Lixing. an~Zhou, Computing tensor eigenvalues via homotopy
  methods, SIAM.J.Matrix.Anal.A. 37~(1) (2016) 290--319.

\bibitem{Homotopy2}
L.~Chen, L.~Han, H.~Yin, et~al., A homotopy method for computing the largest
  eigenvalue of an irreducible nonnegative tensor., J. Comput. Appl. Math.
  355~(1) (2019) 174--181.

\bibitem{BrauerA}
A.~Brauer, Limits for the characteristic roots of a matrix. {II}., Duke Math.
  J. 14~(1) (1947) 21--26.

\bibitem{OstrowskiI}
A.~Ostrowski., uber das nichtverschwinden einer klasse von determinanten und
  die lokalisierung der charakteristischen wurzeln von matrizen., Compositio
  Math. 9~(3) (1951) 209--226.

\bibitem{Qi2005}
L.~Qi, Eigenvalues of a real supersymmetric tensor, J. Symb. Comput. 40~(6)
  (2005) 1302--1324.

\bibitem{Lim}
{Lek-Heng Lim}, Singular values and eigenvalues of tensors: a variational
  approach, in: 1st IEEE International Workshop on Computational Advances in
  Multi-Sensor Adaptive Processing, 2005., 2005, pp. 129--132.

\bibitem{SJY2013}
J.-Y. Shao, A general product of tensors with applications,
  Linear.Algebra.Appl. 439~(8) (2013) 2350 -- 2366.

\bibitem{LI20141}
C.~Li, F.~Wang, J.~Zhao, Y.~Zhu, Y.~Li, Criterions for the positive
  definiteness of real supersymmetric tensors, J. Comput. Appl. Math. 255
  (2014) 1--14.

\bibitem{ZLP2014}
L.~Zhang, L.~Qi, G.~Zhou, {M}-tensors and some applications,
  SIAM.J.Matrix.Anal.A. 35~(2) (2014) 437--452.

\bibitem{horn2012matrix}
R.~A. Horn, C.~R. Johnson, Matrix analysis, Cambridge university press, 2012.

\bibitem{RevModPhysQE}
R.~Horodecki, P.~Horodecki, M.~Horodecki, K.~Horodecki, Quantum entanglement,
  Rev.Mod.Phys. 81 (2009) 865--942.

\bibitem{PhysRevA.94.042343}
F.~Bohnet-Waldraff, D.~Braun, O.~Giraud, Partial transpose criteria for
  symmetric states, Phys.Rev.A. 94 (2016) 042343.

\bibitem{Varga}
R.~Varga, Ger$\breve{s}$gorin and His Circles, Vol.~36, Springer, 2004.

\bibitem{KCC}
K.~Chang, K.~Pearson, T.~Zhang, Perron-frobenius theorem for nonnegative
  tensors, Commun. Math. Sci. 6~(2) (2008) 507--520.

\bibitem{HSL2014}
S.~Hu, Z.~Huang, L.~Qi, Strictly nonnegative tensors and nonnegative tensor
  partition, Sci. China Math. 57 (2014) 181--195.

\bibitem{MNA2015}
M.~R. Kannan, N.~Shaked-Monderer, A.~Berman, Some properties of strong
  {H}-tensors and general {H}-tensors, Linear.Algebra.Appl. 476 (2015) 42--55.

\bibitem{Parsiad1}
P.~Azimzadeh, E.~Bayraktar, High order bellman equations and weakly chained
  diagonally dominant tensors, SIAM.J.Matrix.Anal.A. 40~(1) (2019) 276--298.

\bibitem{xiong2023multipartite}
L.~Xiong, Q.~Qin, J.~Liu, Z.~Gong, Z.~Jiang, N.-s. Sze, Multipartite strongly
  symmetric states and applications to geometric entanglement and classicality,
  Quantum. Inf. Process. 22~(7) (2023) 291.

\bibitem{stronglysymmetric}
L.~Qi, C.~Xu, Y.~Xu, Nonnegative tensor factorization, completely positive
  tensors, and a hierarchical elimination algorithm, SIAM.J.Matrix.Anal.A.
  35~(4) (2014) 1227--1241.

\bibitem{LI2014N}
C.~Li, Y.~Li, X.~Kong, New eigenvalue inclusion sets for tensors, Numer.
  Linear. Algebra. Appl. 21~(1) (2014) 39--50.

\bibitem{Ljiljana1}
L.~Cvetkovi$\acute{c}$, V.~Kostic, R.~Varga, A new gersgorin-type eigenvalue
  inclusion set, Electron. Trans. Numer. Anal. 18 (2004) 73--80.

\bibitem{Ljiljana}
L.~Cvetkovi$\acute{c}$, H-matrix theory vs. eigenvalue localization, Numer.
  Algorithms. 42~(3-4) (2006) 229--245.

\bibitem{Brualdi}
B.~RA, Matrices eigenvalues, and directed graphs, Linear. Multilinear. Algebra.
  11 (1982) 143--165.

\bibitem{LCQ2016S}
C.~Li, A.~Jiao, Y.~Li, An {S}-type eigenvalue localization set for tensors,
  Linear.Algebra.Appl. 493 (2016) 469--483.

\bibitem{RJZ2016}
R.~Zhao, L.~Gao, Q.~Liu, Y.~Li, Criterions for identifying {H}-tensors, Front.
  Math. China 11 (2016) 661--678.

\end{thebibliography}

\end{document}